%% file: main.tex
\documentclass[preprint,12pt]{elsarticle}

\usepackage{amssymb}
\usepackage{amsthm}

\usepackage{caption}
\usepackage{subcaption}
\usepackage{tikzit}
\usepackage{tkz-graph}
\tikzstyle{vertex}=[circle, fill=black, draw, inner sep=0pt, minimum size=5pt]
\newcommand{\vertex}{\node[vertex]}
\usepackage{float}

\usepackage{amsmath}

\DeclareMathOperator{\thin}{thin}
\DeclareMathOperator{\pthin}{pthin}
\DeclareMathOperator{\lmimw}{lmimw}
\DeclareMathOperator{\pw}{pw}
\DeclareMathOperator{\bw}{bw}

\usepackage{epsfig}

\usepackage[figuresright]{rotating}

\newtheorem{theorem}{Theorem}[section]
\newtheorem{corollary}{Corollary}[section]
\newtheorem{proposition}{Proposition}[section]
\newtheorem{remark}{Remark}[section]
\newtheorem{definition}{Definition}[section]

\makeatletter
\def\ps@pprintTitle{%
  \let\@oddhead\@empty
  \let\@evenhead\@empty
  \let\@oddfoot\@empty
  \let\@evenfoot\@oddfoot
}
\makeatother

\begin{document}

\begin{frontmatter}



\title{Trees with proper thinness 2}


\author[1,2]{Flavia Bonomo-Braberman} \ead{fbonomo@dc.uba.ar}

\author[1,2]{Ignacio Maqueda} \ead{imaqueda@dc.uba.ar} 

\author[3]{Nina Pardal} \ead{n.pardal@hud.ac.uk}

\address[1]{Universidad de Buenos Aires, Facultad de Ciencias Exactas y Naturales. Departamento de Computación, Buenos Aires, Argentina} 

\address[2]{CONICET-Universidad de Buenos Aires, Instituto de Investigación en Ciencias de la Computación (ICC), Buenos Aires, Argentina} 

\address[3]{University of Huddersfield, Queensgate, Huddersfield, United Kingdom}

\begin{abstract}
The proper thinness of a graph is an invariant that generalizes the concept of a proper interval graph. Every graph has a numerical value of proper thinness and the graphs with proper thinness~1 are exactly the proper interval graphs. A graph is proper $k$-thin if its vertices can be ordered in such a way that there is a partition of the vertices into $k$ classes satisfying that for each triple of vertices $r < s < t$, such that there is an edge between $r$ and $t$, it is true that if $r$ and $s$ belong to the same class, then there is an edge between $s$ and $t$, and if $s$ and $t$ belong to the same class, then there is an edge between $r$ and $s$. The proper thinness is the smallest value of $k$ such that the graph is proper $k$-thin. In this work we focus on the calculation of proper thinness for trees. We characterize trees of proper thinness~2, both structurally and by their minimal forbidden induced subgraphs. The characterizations obtained lead to a polynomial-time recognition algorithm. We furthermore show why the structural results obtained for trees of proper thinness~2 cannot be straightforwardly generalized to trees of proper thinness~3.
\end{abstract}

\begin{keyword}
forbidden induced subgraphs \sep proper thinness \sep trees



\end{keyword}

\end{frontmatter}



\section{Introduction}



Interval graphs and proper interval graphs are two of the most studied graph classes, since their introduction, respectively, in~\cite{Haj-int,Benser-int} and~\cite{Rob-uig}. One main reason for this interest is that
many real-world applications involve solving problems on graphs which are either
interval graphs themselves or are related to interval graphs in a natural way. Due to this and to their mathematical
and algorithmic properties, there are several generalizations and graph
parameters motivated by them, some of them already surveyed in~\cite{Golumbic-int,Fishburn-book}, and some others (very) recently defined~\cite{B-H-T-mp,Milanic-sim,H-int-wg}.  Among the graph width parameters that, in some sense, measure how close a graph is to being an interval graph, one finds thinness. This parameter was introduced by Mannino, Oriolo, Ricci, and Chandran in 2007 in the context of applications to cellular telephony, specifically related to the independent set problem~\cite{M-O-R-C-thinness}. Subsequently, thinness was also applied in the field of logistics~\cite{B-M-O-thin-tcs}.

A graph $G = (V, E)$ is \emph{$k$-thin} if there exists an ordering of the vertices $v_1, v_2, \dots, v_n$ and a partition $V^1, \dots, V^k$ of $V$ into $k$ classes such that for every triple $(r, s, t)$ with $r < s < t$, if $v_r, v_s$ are in the same class and $v_t v_r \in E$, then $v_t v_s \in E$.


The definition of 1-thin graphs is known to be a characterization of interval graphs~\cite{Ola-interval,R-PR-interval}, and the minimum $k$ such that $G$ is $k$-thin is called the \emph{thinness} of $G$ and is denoted as $\thin(G)$. When an ordering and a partition satisfy the definition of $k$-thin, we say that the ordering and the partition are \emph{consistent}. 

Most of the aforementioned parameters generalizing interval graphs admit a proper version (e.g.~\cite{CFGKZ21}), even if not always such a parameter was defined or studied. In~\cite{B-D-thinness}, a metatheorem-like result for thinness was obtained, i.e., a description of a large family of problems that can be solved in polynomial time on graphs of bounded thinness. Furthermore, the concept of \emph{proper thinness} was defined there, as well as a larger family of problems that can be solved in polynomial time on graphs of bounded proper thinness. 


A graph $G = (V, E)$ is \emph{proper $k$-thin} if there exists an ordering of the vertices $v_1, v_2, \dots, v_n$ and a partition $V^1, \dots, V^k$ of $V$ into $k$ classes such that for every triple $(r, s, t)$ with $r < s < t$ it holds that:
\begin{itemize}
\item If $v_r, v_s$ belong to the same class and $v_t v_r \in E$, then $v_t v_s \in E$
\item If $v_s, v_t$ belong to the same class and $v_t v_r \in E$, then $v_r v_s \in E$
\end{itemize}

Equivalently, both the ordering and its reverse are consistent with the partition. In this case, we say that they are \emph{strongly consistent}. 
The definition of proper 1-thin graphs is known to be a characterization of proper interval graphs~\cite{Rob-uig}, and the \emph{proper thinness} of a graph $G$, denoted as $\pthin(G)$, is the smallest number $k$ for which the graph is proper $k$-thin.  Examples of 3-thin and proper 3-thin graphs are depicted in Figure~\ref{fig:thinness}. 

\begin{figure}[H]
\begin{center}
    \begin{tabular}{ccc}
\input{img/ejemplothinness.tex}  & \hspace{1cm} & \input{img/ejemploproperthinness.tex} 
    \end{tabular}
    \vspace{-0.8cm}
\end{center}
\caption{A $3$-thin representation of a graph (left) and a proper $3$-thin representation of a graph (right). Vertices are ordered from bottom to top, and classes correspond to vertical lines.}\label{fig:thinness}
\end{figure}
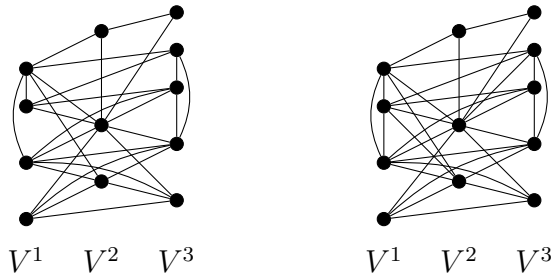

The definitions of thinness and proper thinness gave rise to the variants \emph{independent or complete}~\cite{BGOSS-thin-oper}, \emph{precedence}~\cite{BOSS-thin-prec-dam-lagos}, \emph{mixed or inversion-free mixed}~\cite{Mixed-thin-dm} (and their proper versions). These last variants were defined in 2022 by Balabán, Hlinený, and Jedelský, showing that the parameters have attracted interest in the community in recent years. Regarding the recognition of thinness and proper thinness, it is proved in~\cite{B-M-O-thin-tcs,B-D-thinness} that, given a vertex ordering, there  are $O(n^3)$ algorithms that find, respectively, a consistent or strongly consistent partition into a minimal number of classes. However, given a vertex partition, it is NP-complete to decide whether there exists a (strongly) consistent ordering with that partition~\cite{B-D-thinness}. In~\cite{Shitov-thin}, reducing from that problem, it is shown that deciding whether a graph is $k$-thin is NP-complete. The complexity remains open for fixed $k$ and for the proper variant. Since computing thinness is NP-hard in general, the question is valid for restricted classes of graphs. It has been shown that computing thinness is polynomial for cographs~\cite{B-D-thinness} and for trees~\cite{Agus-Eric-trees-dam}. In both cases, the thinness (and thus the proper thinness) can be arbitrarily large, and the complexity of computing the proper thinness remains open.

The algorithm for thinness of trees is inspired by the algorithm for linear mim-width of trees~\cite{Hogemo-trees-mimw}, which in turn is inspired by the algorithm for pathwidth of trees~\cite{E-S-T-pw-tree,Sch-pw-tree}. It is known~\cite{M-O-R-C-thinness,B-D-thinness} that, in general,

$$\lmimw(G) \leq \thin(G) \leq \pw(G)+1.$$

On the other hand, in~\cite{B-B-lagos21-dam}, it is proved that

$$\pthin(G) \leq \bw(G)+1.$$

However, it is not possible to draw inspiration from bandwidth algorithms in this case, since the problem is NP-complete even for certain fairly restricted subclasses of trees~\cite{Bw-cater} (the long-haired caterpillars with maximum degree~3). In fact, we will see that one of these subclasses of trees is properly contained in the class of trees with proper thinness of at most two.

From a structural point of view, a natural objective that arises is to achieve a characterization by forbidden induced subgraphs of the classes of graphs defined by their thinness or proper thinness. Also in this case, the question is still valid within particular graph classes.
A complete characterization by forbidden induced subgraphs is known for $k$-thin graphs within cographs~\cite{B-D-thinness} and a structural characterization for $k$-thin trees~\cite{Agus-Eric-trees-dam}. In both cases, the problem for proper thinness remains open.

In this work, we study both problems within the class of trees. As a step toward the general case, we provide a characterization that yields a polynomial-time recognition algorithm for trees with proper thinness at most~2.

\subsection{Definitions and basic properties}

All graphs in this work are finite, undirected, and have no loops or multiple edges.

Let $G$ be a graph. Let $V(G)$ denote the set of its vertices and $E(G)$ the set of its edges. For $v\in V(G)$, let $N(v)$ denote its \emph{neighborhood} in $G$, $d(v) = |N(v)|$ its degree, and $N[v] = N(v) \cup \{ v \}$ its \emph{closed neighborhood}. For $X \subseteq V(G)$, let $N(X)$ denote the set of vertices not in $X$ with at least one neighbor in $X$, $N[X] = N(X) \cup X$ its closed neighborhood, $G[X]$ the subgraph of $G$ induced by $X$, and $G - X = G[V(G) \setminus X]$.



If $G_1$ and $G_2$ are two vertex disjoint graphs, the \emph{disjoint union} of $G_1$ and $G_2$, denoted by $G = G_1 \cup G_2$, is such that $V(G) = V(G_1) \cup V(G_2)$ and $E(G) = E(G_1) \cup E(G_2)$.




The \emph{distance} between two vertices $v$ and $w$ in a connected graph $G$ is the length of a shortest path between $v$ and $w$ (the length of a path is measured by the number of edges it comprises) and is denoted $d_G(v,w)$. If the context is unambiguous, it is abbreviated $d(v,w)$.

Given a graph $G$ and an edge $uv \in E(G)$, the graph obtained by the \emph{subdivision} of the edge $uv$ is the graph $G'$
such that $V(G') = V(G) \cup \{w\}$ ($w \not \in V(G)$) and $E(G') = E(G) \setminus \{uv\} \cup \{uw,wv\}$. That is, the edge $uv$ is replaced by a path of length two with endpoints $u$ and $v$. A graph $H$ is said to be \emph{a subdivision} of the graph $G$ if $H$ can be obtained from $G$ by applying a finite number of edge subdivisions. 

A \emph{tree} is a connected graph with no cycles. A \emph{leaf} of a tree $T$ is a vertex of $T$ with degree one.
A tree $T$ is a \emph{caterpillar} if it contains an induced path $P$ such that $V(T) = N[P]$. A \emph{long-haired caterpillar} is a subdivision of a caterpillar. 
A \emph{forest} is a graph with no cycles, that is, the disjoint union of trees.

A graph $H$ is a \emph{forbidden induced subgraph} for a given class of graphs if a graph in that class cannot contain $H$ as an induced subgraph. The forbidden induced subgraph is \emph{minimal} if any proper induced subgraph of it belongs to the class.


An \emph{interval graph} is the intersection graph of intervals on a line. A \emph{proper interval graph} is an interval graph admitting an intersection model in which no interval is properly contained in another.

Interval graphs, and therefore 1-thin graphs, can be recognized in polynomial time~\cite{B-L-PQtrees} and were characterized by forbidden induced subgraphs~\cite{L-B-interval-AT}.

\begin{theorem}\emph{\cite{L-B-interval-AT}}\label{th:LekkerkerkerBoland}
The minimal forbidden induced subgraphs for the class of interval graphs are: bipartite claw, $n$-net for $n\geq 2$, umbrella,
$n$-tent for $n\geq 3$, and $C_n$ for $n\geq 4$ (Figure~\ref{fig:LekkerkerkerBoland}).
\end{theorem}

\begin{figure}[ht]
\begin{center} 
\epsfig{file=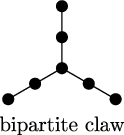}\quad\epsfig{file=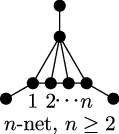}\quad\epsfig{file=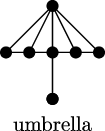} 
\quad\epsfig{file=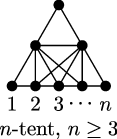}\quad\epsfig{file=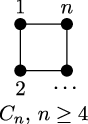}
\caption{Minimal forbidden induced subgraphs for the class of interval graphs.}\label{fig:LekkerkerkerBoland}
\end{center}
\end{figure}

Proper interval graphs, and therefore proper 1-thin graphs, can also be recognized in polynomial time~\cite{F-M-M-prop-int} and were characterized by forbidden induced subgraphs~\cite{Rob-uig}.

\begin{theorem}\emph{\cite{Rob-uig}}\label{th:Roberts}
A graph is a proper interval graph if and only if it is an interval graph and does not have the claw as an induced subgraph. Equivalently, the minimal forbidden induced subgraphs for the class of proper interval graphs are: claw, net, tent, and $C_n$ for $n\geq 4$ (Figure~\ref{fig:Roberts}).
\end{theorem}

\input{img/proh-int-prop}

Notice that since (proper) 1-thin graphs are equivalent to (proper) interval graphs, each class in a vertex partition that is (strongly) consistent with some ordering induces a (proper) interval graph.

\section{Trees with proper thinness~2}

\subsection{Interval and proper interval trees and forests}

It is easy to see that every caterpillar is an interval graph. Furthermore, the \emph{bipartite claw} (see Figure~\ref{fig:LekkerkerkerBoland}), which is not a caterpillar, is known to be one of the minimal forbidden induced subgraphs for the class of interval graphs (Theorem~\ref{th:LekkerkerkerBoland}). Therefore, interval trees are exactly the caterpillars, and interval forests turn out to be the disjoint union of caterpillars.

In turn, the \emph{claw} or $K_{1,3}$ (see Figure~\ref{fig:Roberts}), is the only minimal forbidden induced subgraph for the class of proper interval graphs, within the class of interval graphs (Theorem~\ref{th:Roberts}). Therefore, proper interval trees are just paths, and proper interval forests turn out to be the disjoint union of paths.

Thus, in the partition of a thin representation of a tree, each of the classes is a disjoint union of caterpillars and, for a proper thin representation, a disjoint union of paths.

\subsection{Structural characterization}

In~\cite{Agus-Eric-trees-dam}, it was shown that the following result is valid for the thinness of trees: Given a tree $T$ with thinness equal to $k$, one can always obtain a simple path $C$ such that all the connected components of $T - N[C]$ have thinness less than $k$ and, furthermore, the vertices of $N[C]$ form one of the classes of a $k$-thin partition of $T$.

In this work, we attempt to obtain a similar result for the proper thinness of trees. In this case, taking a simple path $C$ in $T$ and observing the proper thinness of the connected components of $T - C$. We will show that this result is valid for the case of trees of proper thinness~2, that is, there always exists a simple path $C$ such that one can take a partition (which is strongly consistent with some ordering) into two classes: one with the vertices of $C$ and another one with the vertices of $T - C$.

\subsubsection{Structural characterization theorem}

The main result of this work is the following.

\begin{theorem}\label{theorem:characterization_pthin2}
A tree $T$ has proper thinness~2 if and only if the following hold simultaneously:
    \begin{enumerate}
        \item There exists some vertex in $T$ with degree greater than or equal to~3.
        \item If $v$ is a vertex of $T$ with degree greater than or equal to~5, then at most~4 of its neighbors have degree greater than~1.
        \item There exists a simple path $C_0$ in $T$ such that:
        \begin{enumerate}
            \item All vertices of $T$ with degree greater than or equal to~4 are in $C_0$.
            \item If some vertex $v$ of degree~3 is not in $C_0$, then:
            \begin{enumerate}
                \item $v$ is adjacent to a vertex $w$ of $C_0$.
                \item The degree of $w$ is less than or equal to~3. 
            \end{enumerate} 
        \end{enumerate} 
    \end{enumerate}
\end{theorem}

The following corollaries follow from the previous theorem:

\begin{corollary}\label{coro:pthin2_connected_class}
Let $T = (V, E)$ be a tree of proper thinness~2. Then, there exists a simple path $C_0 = v_1, \dots, v_k$ in $T$ and an ordering of $V$ that is strongly consistent with the partition $S = \{ V^0, V^1 \}$, where $V^0 = \{ v_1, \dots, v_k \}$ and $V^1 = V - V^0$.
\end{corollary}

\begin{corollary}\label{coro:pthin2_subtreespthin1}
Let $T = (V, E)$ be a tree of proper thinness~2. Then, there exists a simple path $C_0$ such that all connected components of $T - C_0$ have proper thinness equal to 1.
\end{corollary}

\subsection{Characterization by minimal forbidden induced subgraphs}

From Theorem \ref{theorem:characterization_pthin2}, we also obtained a characterization by minimal forbidden induced subgraphs.

\input{img/forbidden_pt2.tex}

\begin{theorem}\label{theorem:forbidden_subtrees}
    Let $T$ be a tree. Then, $T$ has proper thinness~2 if and only if $T$ has at least one vertex of degree greater than or equal to~3 and $T$ does not contain any of the trees in Figure~\ref{fig:prohibidos_pt2} as an induced subgraph.
\end{theorem}

\begin{proof}
\textbf{$\Rightarrow$)} It is easy to see that each tree in Figure~\ref{fig:prohibidos_pt2} does not satisfy one or more of the properties of the characterization of trees with proper thinness 2. Therefore, if $\pthin(T)=2$, then these trees cannot be induced subgraphs of $T$. Furthermore, since $\pthin(T) = 2$, there must be at least one vertex of degree greater than 2, otherwise $T$ is a path and therefore $\pthin(T) = 1$.

\textbf{$\Leftarrow$)} Suppose that $T$ does not contain any of the trees listed in Figure~\ref{fig:prohibidos_pt2} and that there is a vertex $v \in T$ such that $d(v) \geq 3$, therefore the first property of the characterization holds.
Furthermore, since $T_0$ does not appear as an induced subgraph of $T$, the second property also holds.

Let us see that the third property holds. Suppose that for every simple path $C_0$ of $T$, at least one of statements 3.a, 3.b.i, or 3.b.ii is not true, and we see that a contradiction arises.

Let us first suppose that statement 3.a does not hold. More precisely, for every simple path $C_0$ of $T$, there is a vertex $v \in T$ such that $d(v) \geq 4$ and $v \notin C_0$. Therefore, there are at least three vertices $v_1$, $v_2$, and $v_3$ in $T$ of degree greater than 3 such that there exists a vertex $v_0$ that is the link between those three vertices. Depending on whether $v_0$ is adjacent or not to vertices $v_1$, $v_2$, and $v_3$, this implies that one of the trees in Figure~\ref{fig:prohibidos_pt2} is an induced subgraph of $T$, so we arrive at a contradiction:
\begin{itemize}
\item If $v_0$ is adjacent to all three vertices, $T_1$ is an induced subgraph of $T$.
\item If $v_0$ is adjacent to two of the three vertices, $T_4^i$ is an induced subgraph of $T$ (with $i$ equal to the distance between $v_0$ and the non-adjacent vertex of the three).
\item If $v_0$ is adjacent to one of the three vertices, $T_3^{i,j}$ is an induced subgraph of $T$ (with $i$ and $j$ the distances between $v_0$ and the two non-adjacent vertices).
\item If $v_0$ is not adjacent to any of the three vertices, $T_2^{i,j,k}$ is an induced subgraph of $T$ (with $i$, $j$, and $k$ the distances between $v_0$ and the three vertices).
\end{itemize}

Now, suppose there is a simple path $C_0= v_1, \ldots, v_k$ that contains all vertices of degree greater than 4, but at least one of the statements in 3.b is not true.
Let $v \notin C_0$ be a vertex of degree 3 such that 3.b.i does not hold for $v$. In other words, $v$ is not adjacent to any vertex in $C_0$. We can assume without loss of generality that $v_1$ and $v_k$ are vertices of degree greater than 2.
By hypothesis, there is a vertex $v' \notin C_0$ adjacent to $v$, and since $T$ is connected, there is a path $C_1$ from $v'$ to some vertex $v_i \in C_0$ (which does not pass through $v$).
We can assume that $i \neq 1, k$, otherwise we can consider the union of $C_0$ and $C_1$ as a new path for which the third statement holds.

\begin{itemize}
\item If $d(v_1), d(v_k) = 3$, then there are vertices $w$ and $w'$ in $C_0$, $w \neq v_i \neq w'$, such that $v_1 w, w' v_k \in E(T)$, otherwise we can consider the path joining either $v$ and $v_1$, or $v$ and $v_k$ (i.e., the edge of $C_0$ that has no `intermediate' vertex between itself and $v'$). However, in that case we find $T_2^{i,j,k}$ to be an induced subgraph of $T$.
\item If $d(v_1) = 3$ and $d(v_k) \geq 4$ (w.l.o.g.), then there is a vertex $w$ in $C_0$ ($w \neq v_i$) such that $v_1 w \in E(T)$, otherwise we can consider the path joining $v$ and $v_k$. However, in that case we find $T_3^{i,j}$ to be an induced subgraph of $T$.
\item If $d(v_1) \geq 4$ and $d(v_k) \geq 4$, then:
\begin{itemize}
\item If $v_i$ is adjacent to both $v_1$ and $v_k$, then $T_4^i$ is an induced subgraph of $T$.
\item If $v_i$ is not adjacent to either $v_1$ or $v_k$, $T_2^{i,j,k}$ is an induced subgraph of $T$.
\item If not, that is, $v_i$ is adjacent to only one of the two vertices, $T_3^{i,j}$ is an induced subgraph of $T$.
\end{itemize}
\end{itemize}

In all cases, we arrive at a contradiction.

Finally, let $v \notin C_0$ be a vertex of degree 3 such that 3.b.ii does not hold for $v$. That is, $v$ is adjacent to some vertex $v_i \in C_0$ such that $d(v_i) \geq 4$. Therefore, there is a vertex $w \notin C_0$ adjacent to $v_i$. We can once again assume that $v_1$ and $v_k$ are vertices of degree at least 3. However, we find $T_5^{i,j}$ to be an induced subgraph, where even $i$ or $j$ could be 0. In other words, there can be no subdivisions between $v_i$ and the ends of $C_0$. In this case, it is not possible to consider alternative paths $v v_i v_{i+1} \dots v_k$ or $v v_i v_{i-1} \dots v_1$ for which all the statements hold.

From the above argument, it follows that there must be a simple path containing all vertices of degree greater than or equal to 4 such that statements 3.b.i and 3.b.ii hold, hence $\pthin(T)=2$.
\end{proof}

\begin{proposition} 
The graphs in Figure~\ref{fig:prohibidos_pt2} are minimal with respect to having proper thinness greater than~2.
\end{proposition}

\subsection{Polynomial-time recognition algorithm}

It is not hard to see that the characterizations of Theorems~\ref{theorem:characterization_pthin2} and~\ref{theorem:forbidden_subtrees} lead to a polynomial-time recognition and certifying algorithm for trees of proper thinness~2. Namely, in polynomial time, either the path $C_0$ is identified or one of the forbidden induced subgraphs can be polynomially detected. Moreover, the proof of Theorem~\ref{theorem:characterization_pthin2}, omitted in this version due to lack of space, is constructive and in the positive case a vertex ordering strongly consistent with a partition into two classes can be constructed in polynomial time. Still, as a future work we aim to design a more efficient recognition and certifying algorithm.  

\section{Trees with proper thinness 3}


We know that for the case of trees with proper thinness equal to 2, the Corollaries \ref{coro:pthin2_connected_class} and \ref{coro:pthin2_subtreespthin1} are valid.

We would like to see if some of these properties can be generalized to the case of trees with proper thinness greater than~2. Namely, to see if given a tree with proper thinness equal to~$n$ (with $n > 2$), there exists a partition into $n$ classes and a strongly consistent ordering such that one of the classes induces a simple path or if, at least, there exists a simple path whose removal leaves a graph with strictly smaller proper thinness.


However, neither of these statements are true for trees with proper thinness equal to 3. This can be proven using the counterexample $T_A$, which consists of three copies of the tree $T_0$ from Theorem \ref{theorem:forbidden_subtrees}, joined by a vertex we call $v_0$.

\input{img/counterexample_base.tex}

\begin{proposition}
$T_A$ has proper thinness equal to 3.
\end{proposition}

\begin{proof}
Since $T_A$ has the graph $T_0$ of Theorem~\ref{theorem:forbidden_subtrees} as an induced subgraph, we know that $\pthin(T_A) \geq 3$.
To prove that $\pthin(T_A) \leq 3$, we show a proper $3$-thin representation of $T_A$ (Figure~\ref{fig:T_A_pthin3}). \end{proof}

\input{img/counterexample_representation.tex}

\begin{remark}\label{obs:one_in_each_class}
Since $T_A$ consists of three copies of $T_0$ joined by a new vertex $v_0$ and we already proved that $pthin(T_A) = 3$, we have that for every partition $S = \{ V^0, V^1, V^2 \}$ of the vertices of $T_A$ that is strongly consistent with some vertex ordering, each of those three subtrees has at least one vertex in each of the three classes.
\end{remark}

Using the previous remark, we can prove the following.

\begin{proposition}\label{prop:path1}
Let $S = \{V^0, V^1, V^2\}$ be a class partition of the vertices of $T_A$ and let $\sigma$ be an ordering of the vertices such that $S$ and $\sigma$ are strongly consistent. Then, the induced subgraphs $T_A[V^0]$, $T_A[V^1]$ and $T_A[V^2]$ are not connected.
Moreover, for every simple path $C_0$ of $T_A$, the subgraph $T_A - C_0$ has proper thinness equal to~3.
\end{proposition}

\section*{Acknowledgments}

This work was partially supported by UBACyT (20020220300079BA and 20020190100126BA), CONICET (PIP 11220200100084CO), and ANPCyT (PICT-2021-I-A-00755). 

\bibliographystyle{elsarticle-num}
\bibliography{bnm-jabbr,bnm}

\newpage

\appendix

\section{Proofs}

\input{src/appendix}









\end{document}

%% file: img/ejemplothinness.tex
\begin{tikzpicture}[scale=0.5]
    \vertex(v1) at (-6,4) {};
    \vertex(v2) at (-6,5.5) {};
    \vertex(v3) at (-6,7) {};
    \vertex(v4) at (-6,8) {};
    
    \vertex(w1) at (-4,5) {};
    \vertex(w2) at (-4,6.5) {};
    \vertex(w3) at (-4,9) {};
    
    \vertex(z1) at (-2,4.5) {};
    \vertex(z2) at (-2,6) {};
    \vertex(z3) at (-2,7.5) {};
    \vertex(z4) at (-2,8.5) {};
    \vertex(z5) at (-2,9.5) {};

    \node[label=above:{$V^1$}] at (-6,2) {};
    \node[label=above:{$V^2$}] at (-4,2) {};
    \node[label=above:{$V^3$}] at (-2,2) {};

    \draw (v4) to [bend right=25] (v2);
    \draw (v4) to (v3);
    
    \draw (w3) to (w2);

    \draw (z4) to [bend left=25] (z2);
    \draw (z4) to (z3);
    \draw (z3) to (z2);

    \draw (v2) to (w1);
    \draw (v3) to (w2);
    \draw (v4) to (w1);
    \draw (v4) to (w2);
    \draw (v1) to (w2);
    \draw (v4) to (w3);
    \draw (v2) to (w2);
    \draw (v1) to (w1);

    \draw (v1) to (z1);
    \draw (v2) to (z2);
    \draw (v2) to [bend left=15] (z1);
    \draw (v1) to [bend left=15] (z2);
    \draw (v2) to [bend left=15] (z3);
    \draw (v4) to (z4);
    \draw (v3) to (z4);
    \draw (v3) to (z3);
    
    \draw (z2) to (w1);
    \draw (z3) to (w2);
    \draw (z5) to (w3);
    \draw (z5) to (w2);
    \draw (z2) to (w2);
    \draw (z1) to (w1);
    \draw (z1) to (w2);

\end{tikzpicture}

%% file: img/ejemploproperthinness.tex
\begin{tikzpicture}[scale=0.5]
    \vertex(v1) at (-6,4) {};
    \vertex(v2) at (-6,5.5) {};
    \vertex(v3) at (-6,7) {};
    \vertex(v4) at (-6,8) {};
    
    \vertex(w1) at (-4,5) {};
    \vertex(w2) at (-4,6.5) {};
    \vertex(w3) at (-4,9) {};
    
    \vertex(z1) at (-2,4.5) {};
    \vertex(z2) at (-2,6) {};
    \vertex(z3) at (-2,7.5) {};
    \vertex(z4) at (-2,8.5) {};
    \vertex(z5) at (-2,9.5) {};

    \node[label=above:{$V^1$}] at (-6,2) {};
    \node[label=above:{$V^2$}] at (-4,2) {};
    \node[label=above:{$V^3$}] at (-2,2) {};

    \draw (v4) to [bend right=25] (v2);
    \draw (v4) to (v3);
    \draw (v2) to (v3);
    
    \draw (w3) to (w2);

    \draw (z4) to [bend left=25] (z2);
    \draw (z4) to (z3);
    \draw (z3) to (z2);

    \draw (v2) to (w1);
    \draw (v3) to (w2);
    \draw (v4) to (w1);
    \draw (v4) to (w2);
    \draw (v1) to (w2);
    \draw (v4) to (w3);
    \draw (v2) to (w2);
    \draw (v1) to (w1);
    \draw (v3) to (w1);

    \draw (v1) to (z1);
    \draw (v2) to (z2);
    \draw (v2) to [bend left=15] (z1);
    \draw (v1) to [bend left=15] (z2);
    \draw (v2) to [bend left=15] (z3);
    \draw (v4) to (z4);
    \draw (v3) to (z4);
    \draw (v3) to (z3);
    
    \draw (z2) to (w1);
    \draw (z3) to (w2);
    \draw (z5) to (w3);
    \draw (z5) to (w2);
    \draw (z2) to (w2);
    \draw (z1) to (w1);
    \draw (z1) to (w2);
    \draw (z4) to (w2);

\end{tikzpicture}

%% file: img/proh-int-prop.tex
\begin{figure}[H]
         \centering
          \begin{tikzpicture}[scale=0.65]
        	\begin{pgfonlayer}{nodelayer}
        		\node [circle,fill=black,scale=0.55] (0) at (1.4,2.2) {};
        		\node [circle,fill=black,scale=0.55] (1) at (1.4,3.3) {};
        		\node [circle,fill=black,scale=0.55] (3) at (0.4,1.64) {};
        		\node [circle,fill=black,scale=0.55] (5) at (2.4, 1.64) {};

        		\node [circle,fill=black,scale=0.55] (00) at (5.4,2.75) {};
        		\node [circle,fill=black,scale=0.55] (22) at (4.9,1.92) {};
        		\node [circle,fill=black,scale=0.55] (44) at (5.9,1.92) {};
        		\node [circle,fill=black,scale=0.55] (11) at (5.4,3.3) {};
        		\node [circle,fill=black,scale=0.55] (33) at (4.4,1.64) {};
        		\node [circle,fill=black,scale=0.55] (55) at (6.4, 1.64) {};

        	\node [circle,fill=black,scale=0.55] (000) at (9.4,1.64) {};
        	\node [circle,fill=black,scale=0.55] (222) at (8.9,2.47) {};
        	\node [circle,fill=black,scale=0.55] (444) at (9.9,2.47) {};
        		\node [circle,fill=black,scale=0.55] (111) at (9.4,3.3) {};
        		\node [circle,fill=black,scale=0.55] (333) at (8.4,1.64) {};
        		\node [circle,fill=black,scale=0.55] (555) at (10.4, 1.64) {};

         \end{pgfonlayer}
        	\begin{pgfonlayer}{edgelayer}
        		\draw (0.center) to (1.center);
        		\draw (0.center) to (3.center);
        		\draw (0.center) to (5.center);
        		\draw (111.center) to (333.center);
        		\draw (111.center) to (555.center);
        		\draw (333.center) to (555.center);
                    \draw (000.center) to (222.center);
        		\draw (000.center) to (444.center);
        		\draw (222.center) to (444.center);
                     \draw (11.center) to (00.center);
        		\draw (33.center) to (22.center);
        		\draw (55.center) to (44.center);
                    \draw (00.center) to (22.center);
        		\draw (00.center) to (44.center);
        		\draw (22.center) to (44.center);
        	\end{pgfonlayer}
         \node[align=center] (label1) at (1.4,0.8) {claw};
         \node[align=center] (label2) at (5.4,0.8) {net};
         \node[align=center] (label3) at (9.4,0.8) {tent};
        \end{tikzpicture}
     \hspace{1cm}\epsfig{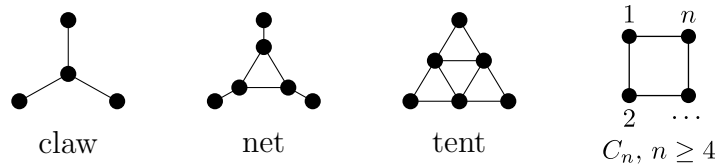}
     
        \caption{Minimal forbidden induced subgraphs for the class of proper interval graphs}
        \label{fig:Roberts}
\end{figure}

%% file: img/forbidden_pt2.tex
\begin{figure}[H]
     \centering
     \begin{subfigure}[b]{0.4\textwidth}
         \centering
         \begin{tikzpicture}[scale=0.3]
        	\begin{pgfonlayer}{nodelayer}
        		\node [circle,fill=black,scale=0.55] (0) at (0, 0) {};
        		\node [circle,fill=black,scale=0.55] (1) at (-1.25, -1.75) {};
        		\node [circle,fill=black,scale=0.55] (2) at (1.25, -1.75) {};
        		\node [circle,fill=black,scale=0.55] (3) at (2.25, 0.75) {};
        		\node [circle,fill=black,scale=0.55] (4) at (-2.25, 0.75) {};
        		\node [circle,fill=black,scale=0.55] (5) at (0, 2.25) {};
        		\node [circle,fill=black,scale=0.55] (6) at (-4.5, 1.5) {};
        		\node [circle,fill=black,scale=0.55] (7) at (0, 4.5) {};
        		\node [circle,fill=black,scale=0.55] (8) at (-2.5, -3.5) {};
        		\node [circle,fill=black,scale=0.55] (9) at (2.5, -3.5) {};
        		\node [circle,fill=black,scale=0.55] (10) at (4.5, 1.5) {};
        	\end{pgfonlayer}
        	\begin{pgfonlayer}{edgelayer}
        		\draw (0.center) to (5.center);
        		\draw (5.center) to (7.center);
        		\draw (0.center) to (4.center);
        		\draw (4.center) to (6.center);
        		\draw (0.center) to (1.center);
        		\draw (1.center) to (8.center);
        		\draw (0.center) to (2.center);
        		\draw (2.center) to (9.center);
        		\draw (0.center) to (3.center);
        		\draw (3.center) to (10.center);
        	\end{pgfonlayer}
         \end{tikzpicture}
         \caption{$T_0$}
         \label{fig:T0}
     \end{subfigure}
     \hfill
     \begin{subfigure}[b]{0.4\textwidth}
         \centering
         \begin{tikzpicture}[scale=0.3]
        	\begin{pgfonlayer}{nodelayer}
        		\node [circle,fill=black,scale=0.55] (0) at (0, 0) {};
        		\node [circle,fill=black,scale=0.55] (1) at (4, 0) {};
        		\node [circle,fill=black,scale=0.55] (2) at (0, -4) {};
        		\node [circle,fill=black,scale=0.55] (3) at (-4, 0) {};
        		\node [circle,fill=black,scale=0.55] (4) at (-6, 0) {};
        		\node [circle,fill=black,scale=0.55] (5) at (-5.5, -1.5) {};
        		\node [circle,fill=black,scale=0.55] (6) at (-5.5, 1.5) {};
        		\node [circle,fill=black,scale=0.55] (7) at (6, 0) {};
        		\node [circle,fill=black,scale=0.55] (8) at (5.5, 1.5) {};
        		\node [circle,fill=black,scale=0.55] (9) at (5.5, -1.5) {};
        		\node [circle,fill=black,scale=0.55] (10) at (0, -6) {};
        		\node [circle,fill=black,scale=0.55] (11) at (1.5, -5.5) {};
        		\node [circle,fill=black,scale=0.55] (12) at (-1.5, -5.5) {};
        	\end{pgfonlayer}
        	\begin{pgfonlayer}{edgelayer}
        		\draw (0.center) to (3.center);
        		\draw (3.center) to (4.center);
        		\draw (6.center) to (3.center);
        		\draw (3.center) to (5.center);
        		\draw (0.center) to (2.center);
        		\draw (2.center) to (12.center);
        		\draw (2.center) to (10.center);
        		\draw (2.center) to (11.center);
        		\draw (0.center) to (1.center);
        		\draw (1.center) to (8.center);
        		\draw (1.center) to (7.center);
        		\draw (1.center) to (9.center);
        	\end{pgfonlayer}
        \end{tikzpicture}

         \caption{$T_1$}
         \label{fig:T1}
     \end{subfigure}
     \hfill
     \begin{subfigure}[b]{0.4\textwidth}
         \centering
         \begin{tikzpicture}[scale=0.3]
        	\begin{pgfonlayer}{nodelayer}
        		\node [circle,fill=black,scale=0.55] (0) at (0, 0) {};
        		\node [circle,fill=black,scale=0.55] (1) at (-2, 0) {};
        		\node [circle,fill=black,scale=0.55] (2) at (0, -2) {};
        		\node [circle,fill=black,scale=0.55] (3) at (2, 0) {};
        		\node [circle,fill=black,scale=0.55] (4) at (4, 0) {};
        		\node [circle,fill=black,scale=0.55] (6) at (-4, 0) {};
        		\node [circle,fill=black,scale=0.55] (7) at (-5.5, 1.5) {};
        		\node [circle,fill=black,scale=0.55] (8) at (-5.5, -1.5) {};
        		\node [circle,fill=black,scale=0.55] (9) at (5.5, 1.5) {};
        		\node [circle,fill=black,scale=0.55] (10) at (5.5, -1.5) {};
        		\node [circle,fill=black,scale=0.55] (11) at (0, -4) {};
        		\node [circle,fill=black,scale=0.55] (12) at (-1.5, -5.5) {};
        		\node [circle,fill=black,scale=0.55] (13) at (1.5, -5.5) {};
        	\end{pgfonlayer}
        	\begin{pgfonlayer}{edgelayer}
        		\draw (0.center) to (1.center);
        		\draw [dashed] (1.center) to (6.center);
        		\draw (6.center) to (7.center);
        		\draw (6.center) to (8.center);
        		\draw (0.center) to (2.center);
        		\draw [dashed] (2.center) to (11.center);
        		\draw (11.center) to (12.center);
        		\draw (11.center) to (13.center);
        		\draw (0.center) to (3.center);
        		\draw [dashed] (3.center) to (4.center);
        		\draw (4.center) to (9.center);
        		\draw (4.center) to (10.center);
        	\end{pgfonlayer}
        \end{tikzpicture}
         \caption{$T_2^{i,j,k}$}
         \label{fig:T2}
     \end{subfigure}
     \hfill
     \begin{subfigure}[b]{0.4\textwidth}
         \centering
         \begin{tikzpicture}[scale=0.3]
        	\begin{pgfonlayer}{nodelayer}
        		\node [circle,fill=black,scale=0.55] (0) at (0, 0) {};
        		\node [circle,fill=black,scale=0.55] (1) at (0, -4) {};
        		\node [circle,fill=black,scale=0.55] (2) at (0, -6) {};
        		\node [circle,fill=black,scale=0.55] (3) at (-1.5, -5.5) {};
        		\node [circle,fill=black,scale=0.55] (4) at (1.5, -5.5) {};
        		\node [circle,fill=black,scale=0.55] (5) at (2, 0) {};
        		\node [circle,fill=black,scale=0.55] (6) at (-2, 0) {};
        		\node [circle,fill=black,scale=0.55] (7) at (4, 0) {};
        		\node [circle,fill=black,scale=0.55] (8) at (5.5, 1.5) {};
        		\node [circle,fill=black,scale=0.55] (9) at (5.5, -1.5) {};
        		\node [circle,fill=black,scale=0.55] (10) at (-4, 0) {};
        		\node [circle,fill=black,scale=0.55] (11) at (-5.5, 1.5) {};
        		\node [circle,fill=black,scale=0.55] (12) at (-5.5, -1.5) {};
        	\end{pgfonlayer}
        	\begin{pgfonlayer}{edgelayer}
        		\draw (0.center) to (6.center);
        		\draw [dashed] (6.center) to (10.center);
        		\draw (10.center) to (11.center);
        		\draw (10.center) to (12.center);
        		\draw (0.center) to (1.center);
        		\draw (1.center) to (3.center);
        		\draw (1.center) to (2.center);
        		\draw (1.center) to (4.center);
        		\draw (0.center) to (5.center);
        		\draw [dashed] (5.center) to (7.center);
        		\draw (7.center) to (8.center);
        		\draw (7.center) to (9.center);
        	\end{pgfonlayer}
        \end{tikzpicture}

         \caption{$T_3^{i,j}$}
         \label{fig:T3}
     \end{subfigure}
     \hfill
     \begin{subfigure}[b]{0.4\textwidth}
         \centering
         \begin{tikzpicture}[scale=0.3]
        	\begin{pgfonlayer}{nodelayer}
        		\node [circle,fill=black,scale=0.55] (0) at (0, 0) {};
        		\node [circle,fill=black,scale=0.55] (1) at (4, 0) {};
        		\node [circle,fill=black,scale=0.55] (2) at (6, 0) {};
        		\node [circle,fill=black,scale=0.55] (3) at (5.5, 1.5) {};
        		\node [circle,fill=black,scale=0.55] (4) at (5.5, -1.5) {};
        		\node [circle,fill=black,scale=0.55] (5) at (-4, 0) {};
        		\node [circle,fill=black,scale=0.55] (6) at (-6, 0) {};
        		\node [circle,fill=black,scale=0.55] (7) at (-5.5, 1.5) {};
        		\node [circle,fill=black,scale=0.55] (8) at (-5.5, -1.5) {};
        		\node [circle,fill=black,scale=0.55] (9) at (0, -2) {};
        		\node [circle,fill=black,scale=0.55] (10) at (0, -4) {};
        		\node [circle,fill=black,scale=0.55] (11) at (-1.5, -5.5) {};
        		\node [circle,fill=black,scale=0.55] (12) at (1.5, -5.5) {};
        	\end{pgfonlayer}
        	\begin{pgfonlayer}{edgelayer}
        		\draw (0.center) to (5.center);
        		\draw (5.center) to (7.center);
        		\draw (6.center) to (5.center);
        		\draw (5.center) to (8.center);
        		\draw (9.center) to (0.center);
        		\draw [dashed] (9.center) to (10.center);
        		\draw (10.center) to (11.center);
        		\draw (10.center) to (12.center);
        		\draw (0.center) to (1.center);
        		\draw (1.center) to (3.center);
        		\draw (1.center) to (2.center);
        		\draw (1.center) to (4.center);
        	\end{pgfonlayer}
        \end{tikzpicture}

         \caption{$T_4^i$}
         \label{fig:T4}
     \end{subfigure}
     \hfill
     \begin{subfigure}[b]{0.4\textwidth}
         \centering
         \begin{tikzpicture}[scale=0.3]
        	\begin{pgfonlayer}{nodelayer}
        		\node [circle,fill=black,scale=0.55] (0) at (0, 0) {};
        		\node [circle,fill=black,scale=0.55] (1) at (4, 0) {};
        		\node [circle,fill=black,scale=0.55] (2) at (5.5, 1.5) {};
        		\node [circle,fill=black,scale=0.55] (3) at (5.5, -1.5) {};
        		\node [circle,fill=black,scale=0.55] (4) at (-4, 0) {};
        		\node [circle,fill=black,scale=0.55] (5) at (-5.5, 1.5) {};
        		\node [circle,fill=black,scale=0.55] (6) at (-5.5, -1.5) {};
        		\node [circle,fill=black,scale=0.55] (7) at (0, -2) {};
        		\node [circle,fill=black,scale=0.55] (8) at (-1.5, -3.5) {};
        		\node [circle,fill=black,scale=0.55] (9) at (1.5, -3.5) {};
        		\node [circle,fill=black,scale=0.55] (10) at (0, 2) {};
        	\end{pgfonlayer}
        	\begin{pgfonlayer}{edgelayer}
        		\draw [dashed] (0.center) to (4.center);
        		\draw (5.center) to (4.center);
        		\draw (4.center) to (6.center);
        		\draw (0.center) to (7.center);
        		\draw (7.center) to (8.center);
        		\draw (7.center) to (9.center);
        		\draw [dashed] (0.center) to (1.center);
        		\draw (1.center) to (2.center);
        		\draw (1.center) to (3.center);
        		\draw (0.center) to (10.center);
        	\end{pgfonlayer}
        \end{tikzpicture}

         \caption{$T_5^{i,j}$}
         \label{fig:T5}
     \end{subfigure}

        \caption{Minimal forbidden induced subgraphs of trees of proper thinness~2. The dashed lines represent possibly subdivided edges.}
        \label{fig:prohibidos_pt2}
\end{figure}
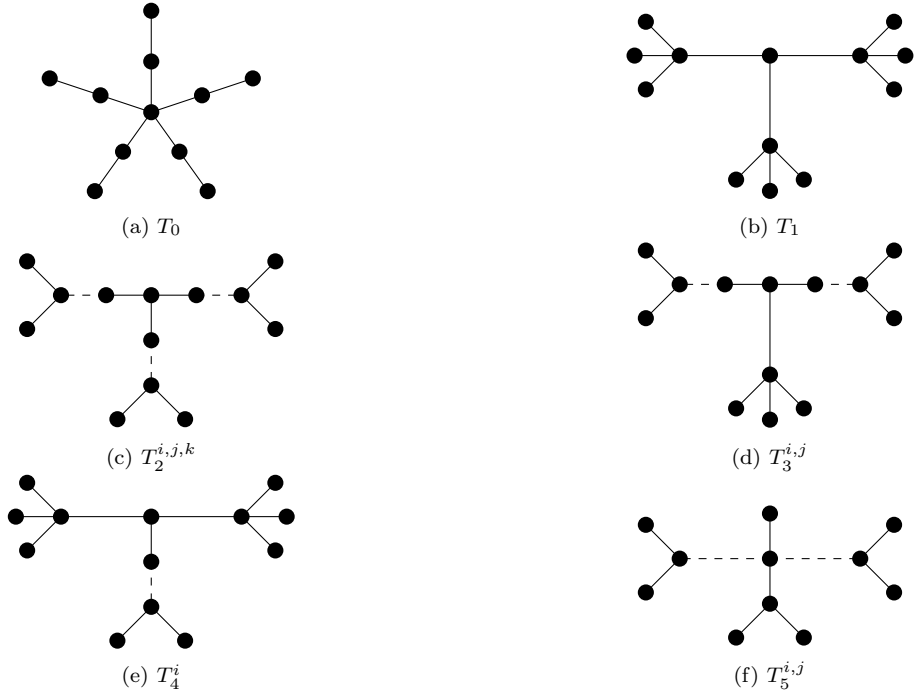

%% file: img/counterexample_base.tex
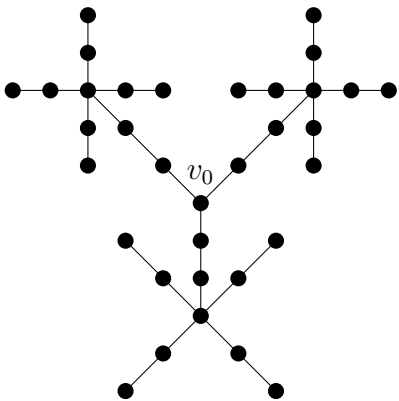
\begin{figure}[H]
    \centering
    \begin{tikzpicture}[scale=0.5]
    	\begin{pgfonlayer}{nodelayer}
    		\node [label={$v_0$}] [circle,fill=black,scale=0.55] (0) at (0, 0) {};
    		\node [circle,fill=black,scale=0.55] (1) at (-1, 1) {};
    		\node [circle,fill=black,scale=0.55] (2) at (1, 1) {};
    		\node [circle,fill=black,scale=0.55] (3) at (0, -1) {};
    		\node [circle,fill=black,scale=0.55] (5) at (-2, 2) {};
    		\node [circle,fill=black,scale=0.55] (6) at (-3, 3) {};
    		\node [circle,fill=black,scale=0.55] (7) at (0, -2) {};
    		\node [circle,fill=black,scale=0.55] (8) at (0, -3) {};
    		\node [circle,fill=black,scale=0.55] (9) at (2, 2) {};
    		\node [circle,fill=black,scale=0.55] (10) at (3, 3) {};
    		\node [circle,fill=black,scale=0.55] (11) at (2, 3) {};
    		\node [circle,fill=black,scale=0.55] (12) at (1, 3) {};
    		\node [circle,fill=black,scale=0.55] (13) at (3, 4) {};
    		\node [circle,fill=black,scale=0.55] (14) at (3, 5) {};
    		\node [circle,fill=black,scale=0.55] (15) at (4, 3) {};
    		\node [circle,fill=black,scale=0.55] (16) at (5, 3) {};
    		\node [circle,fill=black,scale=0.55] (17) at (3, 2) {};
    		\node [circle,fill=black,scale=0.55] (18) at (3, 1) {};
    		\node [circle,fill=black,scale=0.55] (19) at (-2, 3) {};
    		\node [circle,fill=black,scale=0.55] (20) at (-1, 3) {};
    		\node [circle,fill=black,scale=0.55] (21) at (-3, 4) {};
    		\node [circle,fill=black,scale=0.55] (22) at (-3, 5) {};
    		\node [circle,fill=black,scale=0.55] (23) at (-4, 3) {};
    		\node [circle,fill=black,scale=0.55] (24) at (-5, 3) {};
    		\node [circle,fill=black,scale=0.55] (25) at (-3, 2) {};
    		\node [circle,fill=black,scale=0.55] (26) at (-3, 1) {};
    		\node [circle,fill=black,scale=0.55] (27) at (-1, -2) {};
    		\node [circle,fill=black,scale=0.55] (28) at (-2, -1) {};
    		\node [circle,fill=black,scale=0.55] (29) at (-1, -4) {};
    		\node [circle,fill=black,scale=0.55] (30) at (-2, -5) {};
    		\node [circle,fill=black,scale=0.55] (31) at (1, -4) {};
    		\node [circle,fill=black,scale=0.55] (32) at (2, -5) {};
    		\node [circle,fill=black,scale=0.55] (33) at (1, -2) {};
    		\node [circle,fill=black,scale=0.55] (34) at (2, -1) {};
    	\end{pgfonlayer}
    	\begin{pgfonlayer}{edgelayer}
    		\draw (0.center) to (2.center);
    		\draw (2.center) to (9.center);
    		\draw (9.center) to (10.center);
    		\draw (10.center) to (17.center);
    		\draw (17.center) to (18.center);
    		\draw (10.center) to (11.center);
    		\draw (11.center) to (12.center);
    		\draw (10.center) to (13.center);
    		\draw (13.center) to (14.center);
    		\draw (10.center) to (15.center);
    		\draw (15.center) to (16.center);
    		\draw (0.center) to (1.center);
    		\draw (1.center) to (5.center);
    		\draw (5.center) to (6.center);
    		\draw (6.center) to (19.center);
    		\draw (19.center) to (20.center);
    		\draw (6.center) to (25.center);
    		\draw (25.center) to (26.center);
    		\draw (6.center) to (23.center);
    		\draw (23.center) to (24.center);
    		\draw (6.center) to (21.center);
    		\draw (21.center) to (22.center);
    		\draw (0.center) to (3.center);
    		\draw (3.center) to (7.center);
    		\draw (7.center) to (8.center);
    		\draw (8.center) to (33.center);
    		\draw (33.center) to (34.center);
    		\draw (8.center) to (31.center);
    		\draw (31.center) to (32.center);
    		\draw (8.center) to (29.center);
    		\draw (29.center) to (30.center);
    		\draw (27.center) to (8.center);
    		\draw (28.center) to (27.center);
    	\end{pgfonlayer}
    \end{tikzpicture}
    \caption{The tree $T_A$.}
    \label{fig:T_A_base}
\end{figure}

%% file: img/counterexample_representation.tex
\begin{figure}[H]
        \centering
        \begin{tikzpicture}[scale=0.9,rotate = 270]
        	\begin{pgfonlayer}{nodelayer}
        		\node [label={$v_0$}] [circle,fill=black,scale=0.55] (0) at (-1.25, -3.5) {};
        		\node [circle,fill=black,scale=0.55] (1) at (1.25, -1.5) {};
        		\node [circle,fill=black,scale=0.55] (2) at (-1.25, -6) {};
        		\node [circle,fill=black,scale=0.55] (3) at (1.25, -2) {};
        		\node [circle,fill=black,scale=0.55] (5) at (1.25, 1.75) {};
        		\node [circle,fill=black,scale=0.55] (6) at (0, 2.25) {};
        		\node [circle,fill=black,scale=0.55] (7) at (1.25, -3.75) {};
        		\node [circle,fill=black,scale=0.55] (8) at (0, -3.25) {};
        		\node [circle,fill=black,scale=0.55] (9) at (-1.25, -7) {};
        		\node [circle,fill=black,scale=0.55] (10) at (0, -7.75) {};
        		\node [circle,fill=black,scale=0.55] (11) at (1.25, -8.25) {};
        		\node [circle,fill=black,scale=0.55] (12) at (1.25, -9.25) {};
        		\node [label={$V^0$}] (40) at (-0.85, -11) {};
        		\node [circle,fill=black,scale=0.55] (13) at (0, -6.75) {};
        		\node [circle,fill=black,scale=0.55] (14) at (0, -5.75) {};
        		\node [circle,fill=black,scale=0.55] (15) at (-1.25, -8.5) {};
        		\node [circle,fill=black,scale=0.55] (16) at (-1.25, -9.5) {};
        		\node [label={$V^2$}] (42) at (1.65, -11) {};
        		\node [circle,fill=black,scale=0.55] (17) at (0, -8.75) {};
        		\node [circle,fill=black,scale=0.55] (18) at (0, -9.75) {};
        		\node [label={$V^1$}] (41) at (0.4, -11) {};
        		\node [circle,fill=black,scale=0.55] (19) at (1.25, 2.75) {};
        		\node [circle,fill=black,scale=0.55] (20) at (1.25, 3.75) {};
        		\node [circle,fill=black,scale=0.55] (21) at (0, 3.25) {};
        		\node [circle,fill=black,scale=0.55] (22) at (0, 4.25) {};
        		\node [circle,fill=black,scale=0.55] (23) at (0, 1.25) {};
        		\node [circle,fill=black,scale=0.55] (24) at (0, 0.5) {};
        		\node [circle,fill=black,scale=0.55] (25) at (-1.25, 3) {};
        		\node [circle,fill=black,scale=0.55] (26) at (-1.25, 4) {};
        		\node [circle,fill=black,scale=0.55] (27) at (-1.25, -1.25) {};
        		\node [circle,fill=black,scale=0.55] (28) at (-1.25, -0.25) {};
        		\node [circle,fill=black,scale=0.55] (29) at (1.25, -4.25) {};
        		\node [circle,fill=black,scale=0.55] (30) at (1.25, -5) {};
        		\node [circle,fill=black,scale=0.55] (31) at (0, -4.5) {};
        		\node [circle,fill=black,scale=0.55] (32) at (0, -5.25) {};
        		\node [circle,fill=black,scale=0.55] (33) at (0, -1) {};
        		\node [circle,fill=black,scale=0.55] (34) at (0, 0) {};
        	\end{pgfonlayer}
        	\begin{pgfonlayer}{edgelayer}
        		\draw (0.center) to (2.center);
        		\draw (2.center) to (9.center);
        		\draw (9.center) to (10.center);
        		\draw (10.center) to (17.center);
        		\draw (17.center) to (18.center);
        		\draw (10.center) to (11.center);
        		\draw (11.center) to (12.center);
        		\draw (10.center) to (13.center);
        		\draw (13.center) to (14.center);
        		\draw (10.center) to (15.center);
        		\draw (15.center) to (16.center);
        		\draw (0.center) to (1.center);
        		\draw (1.center) to (5.center);
        		\draw (5.center) to (6.center);
        		\draw (6.center) to (19.center);
        		\draw (19.center) to (20.center);
        		\draw (6.center) to (25.center);
        		\draw (25.center) to (26.center);
        		\draw (6.center) to (23.center);
        		\draw (23.center) to (24.center);
        		\draw (6.center) to (21.center);
        		\draw (21.center) to (22.center);
        		\draw (0.center) to (3.center);
        		\draw (3.center) to (7.center);
        		\draw (7.center) to (8.center);
        		\draw (8.center) to (33.center);
        		\draw (33.center) to (34.center);
        		\draw (8.center) to (31.center);
        		\draw (31.center) to (32.center);
        		\draw (8.center) to (29.center);
        		\draw (29.center) to (30.center);
        		\draw (27.center) to (8.center);
        		\draw (28.center) to (27.center);
        	\end{pgfonlayer}
        \end{tikzpicture}
        \caption{Proper $3$-thin representation of $T_A$. The vertices are ordered from left to right, and the classes correspond to the horizontal lines.}
        \label{fig:T_A_pthin3}
    \end{figure}
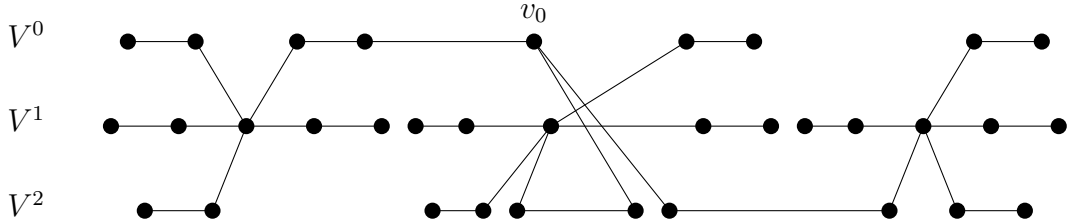

%% file: src/appendix.tex
\subsection{Structural characterization}

\subsubsection{Auxiliary results}

We begin by stating and proving some results that we will need to prove the characterization. We will also make use of two basic properties of thinness and proper thinness, namely: (1) if $H$ is an induced subgraph of $G$, then $\thin(H) \leq \thin(G)$ and $\pthin(H) \leq \pthin(G)$ (straightforward from definition); (2) for two vertex disjoint graphs $G_1$ and $G_2$, $\thin(G_1 \cup G_2) = \thin(G_1) \cup \thin(G_2)$ and $\pthin(G_1 \cup G_2) = \pthin(G_1) \cup \pthin(G_2)$~\cite{B-D-thinness}.

\begin{proposition}\label{prop:same_class_neighbors}
Let $T = (V,E)$ be a tree. Let $\sigma$ be an ordering and $S = \{V^0, \dots , V^k\}$ be a class partition of the vertices such that $\sigma$ and $S$ are strongly consistent. Let $v$ be a vertex of $T$ such that $v \in V^i$. Then there are at most two vertices $u$ and $w$ in $V^i \cap N(v)$, and it holds that $u < v < w$.
\end{proposition}

\begin{proof}
Suppose there are two vertices $u$ and $w$ in $V^i \cap N(v)$ such that $v < u < w$. Since $v$ and $u$ belong to the same class and $v$ is adjacent to $w$, then $u$ is adjacent to $w$, which contradicts $T$ being a tree. Equivalently, there cannot be two vertices $u$ and $w$ in $V^i \cap N(v)$ such that $u < w < v$.
\end{proof}

\begin{proposition}\label{prop:neighbors_other_class}
Let $T = (V,E)$ be a tree with $\pthin(T) = 2$. Let $\sigma$ be an ordering and $S = \{V^0, V^1\}$ be a class partition of the vertices such that $\sigma$ and $S$ are strongly consistent. Let $v$ be a vertex of $T$ such that $v \in V^0$. Let $w_1$, $w_2$, and $w_3$ be three vertices belonging to $N(v) \cap V^1$. If $w_1 < w_2 < w_3$, then $d(w_2) = 1$.
\end{proposition}

\begin{proof}
Suppose there exists a vertex $u \neq v$ such that $u$ is adjacent to $w_2$. Since $T$ is a tree, it holds that $u \neq w_1$ and $u \neq w_3$. We can see that in all cases it turns out that there is an adjacency between some pair of vertices that forms a cycle, so it contradicts the fact that $T$ is a tree:
\begin{itemize}
\item If $u \in V^1$:
\begin{itemize}
\item If $u < w_1$, it holds that $u < w_1 < w_2$, and since $u$ and $w_1$ belong to the same class and $u$ is adjacent to $w_2$, then $w_1$ is adjacent to $w_2$.
\item If $u > w_3$, it holds that $w_2 < w_3 < u$, and since $u$ and $w_3$ belong to the same class and $u$ is adjacent to $w_2$, then $w_2$ is adjacent to $w_3$.
\item If $w_1 < u < w_2$:
\begin{itemize}
\item If $v < u$, then $v < u < w_2$, and since $u$ and $w_2$ belong to the same class and $v$ is adjacent to $w_2$, then $v$ is adjacent to $u$.
\item If $u < v$, then $w_1 < u < v$, and since $w_1$ and $u$ belong to the same class and $v$ is adjacent to $w_1$, then $u$ is adjacent to $v$.
\end{itemize}
\item If $w_2 < u < w_3$:
\begin{itemize}
\item If $v < u$, then $v < u < w_3$, and since $u$ and $w_3$ belong to the same class and $v$ is adjacent to $w_3$, then $v$ is adjacent to $u$.
\item If $u < v$, then $w_2 < u < v$, and since $w_2$ and $u$ belong to the same class and $v$ is adjacent to $w_2$, then $u$ is adjacent to $v$.
\end{itemize}
\end{itemize}
\item If $u \in V^0$:
\begin{itemize}
\item If $u < w_1$, then $u < w_1 < w_2$, and since $w_1$ and $w_2$ belong to the same class and $u$ is adjacent to $w_2$, then $u$ is adjacent to $w_1$.
\item If $u > w_3$, then $w_2 < w_3 < u$, and since $w_2$ and $w_3$ belong to the same class and $u$ is adjacent to $w_2$, then $w_3$ is adjacent to $u$.
\item If $w_1 < u < w_2$:
\begin{itemize}
\item If $v < u$, then $v < u < w_3$, and since $v$ and $u$ belong to the same class and $v$ is adjacent to $w_3$, then $u$ is adjacent to $w_3$.
\item If $u < v$, then $w_1 < u < v$, and since $u$ and $v$ belong to the same class and $w_1$ is adjacent to $v$, then $w_1$ is adjacent to $u$.
\end{itemize}
\item If $w_2 < u < w_3$:
\begin{itemize}
\item If $v < u$, then $v < u < w_3$, and since $u$ and $v$ belong to the same class and $v$ is adjacent to $w_3$, then $u$ is adjacent to $w_3$.
\item If $u < v$, then $w_1 < u < v$, and since $u$ and $v$ belong to the same class and $w_1$ is adjacent to $v$, then $w_1$ is adjacent to $u$
\end{itemize}
\end{itemize}
\end{itemize}
\end{proof}

\begin{definition}\label{def:before_after}
Let $T = (V,E)$ be a tree with $\pthin(T) = 2$. Let $\sigma$ be an ordering and $S = \{V^0, V^1\}$ be a class partition of the vertices such that $\sigma$ and $S$ are strongly consistent. Let $v_0$ and $v_1$ be two adjacent vertices such that $v_0 \in V^0$ and $v_1 \in V^1$. Let $u$ be a vertex other than $v_0$ and $v_1$.
\begin{itemize}
\item If either $u \in V^0$ and $u < v_0$, or $u \in V^1$ and $u < v_1$, we say that $u$ is \textbf{before} the edge $v_0 v_1$ in $\sigma$.
\item If, on the other hand, either $u \in V^0$ and $v_0 < u$, or $u \in V^1$ and $v_1 < u$, we say that $u$ is \textbf{after} the edge $v_0 v_1$ in $\sigma$.
\end{itemize}
\end{definition}

\begin{definition}\label{def:str_before_after}
Let $T = (V,E)$ be a tree with $\pthin(T) = 2$. Let $\sigma$ be an ordering and $S = \{V^0, V^1\}$ be a partition into classes of the vertices such that $\sigma$ and $S$ are strongly consistent. Let $v_0$ and $v_1$ be two adjacent vertices such that $v_0 \in V^0$ and $v_1 \in V^1$. Let $u$ be a vertex distinct from $v_0$ and $v_1$.
\begin{itemize}
\item If $u < v_0$ and $u < v_1$, we say that $u$ is \textbf{strictly before} the edge $v_0 v_1$ in $\sigma$.
\item If $v_0 < u$ and $v_1 < u$, we say that $u$ is \textbf{strictly after} the edge $v_0 v_1$ in $\sigma$.
\end{itemize}
\end{definition}

\begin{proposition}\label{prop:str_before_implies_before}
Let $T = (V,E)$ be a tree with $\pthin(T) = 2$. Let $\sigma$ be an ordering and $S = \{V^0, V^1\}$ be a class partition of the vertices such that $\sigma$ and $S$ are strongly consistent. Let $v_0$ and $v_1$ be two adjacent vertices such that $v_0 \in V^0$ and $v_1 \in V^1$. Let $u$ be a vertex distinct from $v_0$ and $v_1$.
\begin{itemize}
\item If $u$ is strictly before the edge $v_0 v_1$ in $\sigma$, then it is before the edge $v_0 v_1$ in $\sigma$.
\item If $u$ is strictly after the edge $v_0 v_1$ in $\sigma$, then it is after the edge $v_0 v_1$ in $\sigma$.
\end{itemize}
\end{proposition}

\begin{proposition}\label{prop:no_adjacent_implies_str_before}
Let $T = (V,E)$ be a tree with $\pthin(T) = 2$. Let $\sigma$ be an ordering and $S = \{V^0, V^1\}$ be a partition into classes of the vertices such that $\sigma$ and $S$ are strongly consistent. Let $v_0$ and $v_1$ be two adjacent vertices such that $v_0 \in V^0$ and $v_1 \in V^1$.
\begin{itemize}
\item Let $u$ be a vertex that is neither adjacent to $v_0$ nor $v_1$ such that $u < v_0$. Then $u$ is strictly before the edge $v_0 v_1$.
\item Similarly, if $u$ is neither adjacent to $v_0$ nor $v_1$ and $v_0 < u$, then $u$ is strictly after the edge $v_0 v_1$.
\end{itemize}
\end{proposition}

\begin{proof}
Let $u$ be a vertex that is neither adjacent to $v_0$ nor $v_1$ such that $u < v_0$. We want to prove that $u < v_1$. Suppose the inequality does not hold. Then, $v_1 < u < v_0$. But then, since $v_0$ is adjacent to $v_1$:
\begin{itemize}
\item If $u$ and $v_0$ belong to the same class, then $u$ is adjacent to $v_1$.
\item If $u$ and $v_1$ belong to the same class, then $u$ is adjacent to $v_0$.
\end{itemize}
In both cases, a contradiction arises.
\end{proof}

\begin{proposition}\label{prop:does_not_cross_edge}
Let $T = (V,E)$ be a tree with $\pthin(T) = 2$. Let $\sigma$ be an ordering and $S = \{V^0, V^1\}$ be a partition into classes of the vertices such that $\sigma$ and $S$ are strongly consistent. Let $v_0$ and $v_1$ be two adjacent vertices such that $v_0 \in V^0$ and $v_1 \in V^1$. Let $u$ be a vertex that is before the edge $v_0 v_1$ in $\sigma$ and $w$ be a vertex that is after the edge $v_0 v_1$ in $\sigma$. Then there cannot exist a path from $u$ to $w$ that does not pass through either $v_0$ or $v_1$.
\end{proposition}

\begin{proof}
Suppose there exists a path $u = u_0 \dots u_k = w$ such that $u_i \neq v_0$ and $u_i \neq v_1$ $\forall i$. Since $u$ is before $v_0 v_1$ and $w$ is after, there exists some $i$ such that $u_i$ is before $v_0 v_1$ and $u_{i+1}$ is after. We can see that in all cases it turns out that there are adjacencies between the vertices that contradict the fact that $T$ is a tree:
\begin{itemize}
\item If $u_i$ and $u_{i+1}$ belong to the same class, for example, $u_i \in V^0$ and $u_{i+1} \in V^0$, and since $u_i < v_0 < u_{i+1}$, it holds that $v_0$ is adjacent to both $u_i$ and $u_{i+1}$.
\item If $u_i$ and $u_{i+1}$ belong to different classes, for example, $u_i \in V^0$ and $u_{i+1} \in V^1$. Suppose, w.l.o.g., that $u_i < u_{i+1}$:
\begin{itemize}
\item If $u_i < v_0, v_1 < u_{i+1}$, then $v_0$ is adjacent to $u_{i+1}$ and $v_1$ is adjacent to $u_i$.
\item If $v_0 < u_i < u_{i+1} < v_1$ (w.l.o.g.), then $v_0$ is adjacent to $u_{i+1}$ and $v_1$ is adjacent to $u_i$.
\item If $u_i < v_1 < u_{i+1} < v_0$ (w.l.o.g.), then $v_0$ is adjacent to $u_{i+1}$ and $v_1$ is adjacent to $u_i$.
\end{itemize}
\end{itemize}
\end{proof}

\begin{definition}\label{def:nexus}
Let $T = (V,E)$ be a tree. Let $v_0$, $v_1$, $v_2$, and $v_3$ be four distinct vertices of $T$. For $i \in \{1,2,3\}$, we call $C_i$ the unique simple path from $v_0$ to $v_i$, and $C_i' = C_i - \{v_0\}$. We say that vertex $v_0$ is the \textbf{nexus} between vertices $v_1$, $v_2$, and $v_3$ if $C_1' \cap C_2' = \emptyset$, $C_1' \cap C_3' = \emptyset$, and $C_2' \cap C_3' = \emptyset$.
\end{definition}

\begin{proposition}\label{prop:nexus_not_on_path}
Let $T = (V,E)$ be a tree. Let $v_0$, $v_1$, $v_2$, and $v_3$ be four distinct vertices of $T$ such that $v_0$ is the link between $v_1$, $v_2$, and $v_3$. Then, it cannot be the case that there exists a simple path including $v_1$, $v_2$, and $v_3$.
\end{proposition}

\begin{proof}
Suppose that what we want to prove does not hold. Suppose then, w.l.o.g., that the only simple path $C$ from $v_1$ to $v_3$ passes through $v_2$. But then, if there exists a vertex $v_0$ that is the nexus between the 3 vertices, for $i \in {1,2,3}$, there exist paths $C_i'$ from $v_0$ to the other 3 vertices (disjoint from each other). Let $w$ be the first vertex that appears on the path $C_2'$ (going from $v_0$ to $v_2$) such that $w \in C$ (it may happen that $w = v_0$). Suppose (w.l.o.g.) that $w$ is on the subpath between $v_1$ and $v_2$ (it may happen that $w = v_2$). Then, since $v_2 \notin C_3'$, there exist two distinct paths from $v_0$ to $v_3$: one is $C_3$ and the other is the one that passes through $w$ and through $v_2$. But this contradicts that $T$ is a tree.
\end{proof}

\begin{proposition}\label{prop:degree_4_does_not_cross}
Let $T = (V,E)$ be a tree with $\pthin(T) = 2$. Let $\sigma$ be an ordering and $S = \{V^0, V^1\}$ be a partition into classes of the vertices such that $\sigma$ and $S$ are strongly consistent. Let $v_1$, $v_2$, and $v_3$ be three vertices such that $v_1 < v_2 < v_3$. If $d(v_2) \geq 4$, then there cannot exist a vertex $w$ ($w \neq v_1$ and $w \neq v_3$) adjacent to $v_2$ such that the only simple path from $v_2$ to $v_1$ passes through $w$ and the only simple path from $v_2$ to $v_3$ passes through $w$.
\end{proposition}

\begin{proof}
Suppose, w.l.o.g., that $v_2 \in V^0$. Suppose there exists a vertex $w$ adjacent to $v_2$ such that the only simple path from $v_2$ to $v_1$ and the only simple path from $v_2$ to $v_3$ pass through $w$. Since $d(v_2) \geq 4$, by Proposition~\ref{prop:same_class_neighbors} there exists a vertex $x \in V^1$ adjacent to $v_2$ such that $x \neq w$. Furthermore, since $v_1 < v_2$ and they are not adjacent, by Proposition~\ref{prop:no_adjacent_implies_str_before}, $v_1$ is strictly before edge $v_2 x$ (and by Proposition~\ref{prop:str_before_implies_before}, in particular it is before). Similarly, $v_3$ is after edge $v_2 x$. By Proposition~\ref{prop:does_not_cross_edge}, either there cannot exist a simple path from $w$ to $v_1$ that does not pass through $v_2$ or $x$ (if $w$ is after $v_2 x$), or there cannot exist a simple path from $w$ to $v_3$ that does not pass through $v_2$ or $x$ (if $w$ is before $v_2 x$).
\end{proof}

\begin{corollary}\label{coro:degree_4_outside_the_path}
Let $T = (V,E)$ be a tree with $\pthin(T) = 2$. Let $\sigma$ be an ordering and $S = \{V^0, V^1\}$ be a class partition of the vertices such that $\sigma$ and $S$ are strongly consistent. Let $v_1$, $v_2$, and $v_3$ be three vertices such that $v_1 < v_2 < v_3$. If $d(v_2) \geq 4$, then the only simple path from $v_1$ to $v_3$ passes through $v_2$.
\end{corollary}

\begin{proof}
By the above proposition, there cannot exist a vertex $w$ ($w \neq v_1$ and $w \neq v_3$) adjacent to $v_2$ such that the only simple path from $v_2$ to $v_1$ passes through $w$ and the only simple path from $v_2$ to $v_3$ passes through $w$. Then, there exist two vertices $u$ and $v$ ($u \neq v$) adjacent to $v_2$ such that the only simple path from $v_2$ to $v_1$ passes through $u$ and the only simple path from $v_2$ to $v_3$ passes through $v$. Therefore, there is a simple path $v_1 \dots u v_2 v \dots v_3$.
\end{proof}

\begin{corollary}\label{coro:degree_4_no_nexus}
Let $T = (V,E)$ be a tree with $\pthin(T) = 2$. Let $\sigma$ be an ordering and $S = \{V^0, V^1\}$ be a class partition of the vertices such that $\sigma$ and $S$ are strongly consistent. Let $v_1$, $v_2$, and $v_3$ be three vertices such that $v_1 < v_2 < v_3$. If $d(v_2) \geq 4$, then there cannot exist a vertex $v_0$ that is the nexus between vertices $v_1$, $v_2$, and $v_3$.
\end{corollary}

\begin{proof}
By the previous Corollary, let $C$ be the unique simple path from $v_1$ to $v_3$ that passes through $v_2$. But this contradicts Proposition~\ref{prop:nexus_not_on_path}.
\end{proof}

\begin{proposition}\label{prop:degree_3_ady_does_not_cross}
Let $T = (V,E)$ be a tree with $\pthin(T) = 2$. Let $\sigma$ be an ordering and $S = \{V^0, V^1\}$ be a class partition of the vertices such that $\sigma$ and $S$ are strongly consistent. Let $v_1$, $v_2$, and $v_3$ be three vertices such that $v_1 < v_2 < v_3$. If $d(v_2) = 3$, then there cannot exist a vertex $w$ ($w \neq v_1$ and $w \neq v_3$) adjacent to $v_2$ and a vertex $u$ ($u \neq v_1$, $u \neq v_2$ and $u \neq v_3$) adjacent to $w$ such that the only simple path from $v_2$ to $v_1$ passes through $u$ and the only simple path from $v_2$ to $v_3$ passes through $u$.
\end{proposition}

\begin{proof}
Suppose, w.l.o.g., that $v_2 \in V^0$. Suppose that there exist $w$ and $u$.
\begin{itemize}
\item If there exists a vertex $x \in V^1$ adjacent to $v_2$ such that $x \neq w$. Then, since $v_1 < v_2$ and they are not adjacent, by Proposition~\ref{prop:no_adjacent_implies_str_before}, $v_1$ is strictly before edge $v_2 x$ (and by Proposition~\ref{prop:str_before_implies_before}, in particular it is before). Similarly, $v_3$ is after edge $v_2 x$. By Proposition~\ref{prop:does_not_cross_edge}, either there cannot exist a simple path from $w$ to $v_1$ that does not pass through $v_2$ or $x$ (if $w$ is after $v_2 x$), or there cannot exist a simple path from $w$ to $v_3$ that does not pass through $v_2$ or $x$ (if $w$ is before $v_2 x$). \item If not, the only vertex in $V^1$ adjacent to $v_2$ is $w$. Since $v_1 < v_2$ and $v_1$ is not adjacent to either $v_2$ or $w$, by Proposition~\ref{prop:no_adjacent_implies_str_before}, $v_1$ is strictly before edge $v_2 w$ (and by Proposition~\ref{prop:str_before_implies_before}, in particular it is before). Analogously, $v_3$ is after edge $v_2 w$. By Proposition~\ref{prop:does_not_cross_edge}, either there cannot exist a simple path from $u$ to $v_1$ that does not pass through $v_2$ or $w$ (if $u$ is after $v_2 w$), or there cannot exist a simple path from $u$ to $v_3$ that does not pass through $v_2$ or $w$ (if $u$ is before $v_2 w$).
\end{itemize}
\end{proof}

\begin{corollary}\label{coro:degree_3_ady_off_path}
Let $T = (V,E)$ be a tree with $\pthin(T) = 2$. Let $\sigma$ be an ordering and $S = \{V^0, V^1\}$ be a class partition of the vertices such that $\sigma$ and $S$ are strongly consistent. Let $v_1$, $v_2$, and $v_3$ be three vertices such that $v_1 < v_2 < v_3$. If $d(v_2) = 3$, then either the only simple path from $v_1$ to $v_3$ passes through $v_2$, or there exists some vertex $v_0$ on the simple path from $v_1$ to $v_3$ such that $v_0$ is adjacent to $v_2$.
\end{corollary}

\begin{proof}
Let $C_0$ be the unique simple path from $v_1$ to $v_3$. Suppose what we want to prove is not true, that is, $v_2 \notin C_0$ and there is no vertex in $C_0$ that is adjacent to $v_2$. Let $v$ be the vertex in $C_0$ closest to $v_2$. Then, the unique path from $v_2$ to $v_1$ passes through $v$ and the unique simple path passes through $v$. Let $C_2$ be the unique simple path from $v_2$ to $v$. Since $C_2$ has at least 3 vertices, $C_2 = v_2, w, u, \dots$. But then, these vertices $w$ and $u$ satisfy what we said in the previous Proposition could not happen, so we arrive at a contradiction. \end{proof}

\begin{corollary}\label{coro:degree_3_ady_no_nexus}
Let $T = (V,E)$ be a tree with $\pthin(T) = 2$. Let $\sigma$ be an ordering and $S = \{V^0, V^1\}$ be a class partition of the vertices such that $\sigma$ and $S$ are strongly consistent. Let $v_1$, $v_2$, and $v_3$ be three vertices such that $v_1 < v_2 < v_3$. If $d(v_2) = 3$ and there exists a vertex $v_0$ that is the nexus between vertices $v_1$, $v_2$, and $v_3$, then $v_0$ is adjacent to $v_2$.
\end{corollary}

\begin{proof}
By the previous corollary, since $d(v_2) = 3$, there are two possibilities: either $v_2$ belongs to the path between $v_1$ and $v_3$, or it is adjacent to the path. But since there is also a link between $v_1$, $v_2$, and $v_3$, it cannot happen that all three vertices belong to the same path. Therefore, the first option cannot happen, that is, $v_2$ is adjacent to a vertex $w$ of the path. Since $w$ is a link and the link is unique, $w = v_0$.
\end{proof}

\begin{proposition}\label{prop:no_double_crossing}
Let $T = (V,E)$ be a tree with $\pthin(T) = 2$. Let $\sigma$ be an ordering and $S = \{V^0, V^1\}$ be a class partition of the vertices such that $\sigma$ and $S$ are strongly consistent. Let $v_1$, $v_2$, $v_3$, and $v_4$ be four vertices that form a path $v_1, v_2, v_3, v_4$ and that alternate in classes (e.g., $v_1, v_3 \in V^0$ and $v_2, v_4 \in V^1$). Then, neither $v_1 < v_3$ and $v_4 < v_2$ can happen simultaneously, nor can $v_3 < v_1$ and $v_2 < v_4$.
\end{proposition}

\begin{proof}
Suppose that $v_1 < v_3$ and $v_4 < v_2$ both hold. We want to see that a contradiction arises:
\begin{itemize}
\item If $v_1 < v_4$, then since $v_4 < v_2$ also holds, the inequality $v_1 < v_4 < v_2$ holds. And since $v_2$ and $v_4$ are in the same class and $v_1$ is adjacent to $v_2$, then $v_1$ must be adjacent to $v_4$.
\item If $v_4 < v_1$, then since $v_1 < v_3$ also holds, the inequality $v_4 < v_1 < v_3$ holds. And since $v_1$ and $v_3$ are in the same class, and $v_3$ is adjacent to $v_4$, then $v_1$ must be adjacent to $v_4$.
\end{itemize}
In both cases, we conclude that $v_1$ must be adjacent to $v_4$. But then a cycle is formed between the four vertices, which contradicts $T$ as a tree.
\end{proof}

\begin{proposition}\label{prop:pthin2_has_claw}
Let $T$ be a tree such that $\pthin(T) = 2$. Then there exists some vertex in $T$ with degree greater than or equal to~3.
\end{proposition}

\begin{proof}
Let $T = (V,E)$ be a tree with $\pthin(T) = 2$. Suppose what we want to prove does not hold, that is, $d(v) \leq 2$ $\forall v \in V$. Then, since $T$ is a tree, $T$ is a simple path $v_1, \dots, v_n$. But then there exists an ordering $v_1 < \dots < v_n$ that is strongly consistent with $S' = \{V\}$, so $\pthin(T) = 1$.
\end{proof}

\begin{proposition}\label{prop:pthin2_degree_5}
Let $T$ be a tree such that $\pthin(T) = 2$. If $v$ is a vertex of $T$ with degree greater than or equal to~5, then at most~4 of its neighbors have degree greater than~1.
\end{proposition}

\begin{proof}
Let $v$ be a vertex of $T$ with $d(v) \geq 5$. By Proposition~\ref{prop:same_class_neighbors}, there are at most two vertices of $N(v)$ in the same class as $v$. Let $v_1 < \dots < v_l$ be the vertices neighboring $v$ in the other class. By Proposition~\ref{prop:neighbors_other_class}, of all such vertices, only $v_1$ and $v_l$ can have degree greater than~1.
\end{proof}

\begin{proposition}\label{prop:pthin2_has_path}
Let $T$ be a tree such that $\pthin(T) = 2$. Then, there exists a simple path $C_0$ in $T$ such that:
\begin{enumerate}
    \item All vertices of $T$ with degree greater than or equal to~4 are in $C_0$.
    \item If any vertex $v$ of degree~3 is not in $C_0$, then:
    \begin{enumerate}
        \item $v$ is adjacent to a vertex $w$ in $C_0$.
        \item The degree of $w$ is less than or equal to~3.
    \end{enumerate}
\end{enumerate}
\end{proposition}

\begin{proof}
Let $T = (V, E)$ be a tree with $\pthin(T) = 2$. That is, there exists an ordering $\sigma$ and a partition into classes of the vertices $S = \{V^0, V^1\}$ that are  strongly consistent. We want to prove that there exists a simple path $C_0$ satisfying 1. and 2.:
\begin{enumerate}
\item Suppose that 1. does not hold, that is, there exist 3 vertices $v_1$, $v_2$ and $v_3$ with $d(v_1) \geq 4$, $d(v_2) \geq 4$ and $d(v_3) \geq 4$ that cannot be on the same simple path. Suppose, w.l.o.g., that $v_1 < v_2 < v_3$. But by Corollary~\ref{coro:degree_4_outside_the_path} the only simple path from $v_1$ to $v_3$ passes through $v_2$, which contradicts the above. Therefore, there exists a path $C_1$ that contains all vertices in $T$ of degree greater than or equal to~4.
\item
\begin{enumerate}
\item By the above, if $\pthin(T) = 2$, let $C_1 = u_0, \dots, u_k$ be the path that exists through the previous point. Let $W_1 = \{ w_0, \dots, w_l \}$ be the set of vertices of degree~3 that are not in $C_1$ and that are not adjacent to $C_1$. We want to construct a path $C_2$ that also satisfies the condition of the previous point and such that the set $W_2$ of vertices of degree~3 that are not in $C_2$ and that are not adjacent to $C_2$ is contained in $\{ w_1, \dots, w_l \}$. Let $u_i$ be the vertex in $C_2$ such that the simple path from $u_r$ to $w_0$ does not pass through any other vertex in $C_2$. We see that in all cases, one can take $C_2 = u_0, \dots , u_i, \dots, w_0$ or $C_2 = w_0, \dots , u_i, \dots, u_k$. To do this, let us split the cases into subpaths $u_0, \dots, u_{i-1}$ and $u_{i+1}, \dots, u_k$ and see that in all cases one of these two subpaths can be discarded:
\begin{itemize}
\item If both subpaths contain a vertex of degree~4, let $v_1$ and $v_2$ be two vertices of degree~4, one in each subpath.
\begin{itemize}
\item If $v_1 < w_0 < v_2$ or $v_2 < w_0 < v_1$, the Corollary~\ref{coro:degree_3_ady_off_path} contradicts the fact that $w_0$ does not belong to the unique simple path between $v_1$ and $v_2$ nor is it adjacent to a vertex on that path. \item If $w_0 < v_1 < v_2$, $v_2 < v_1 < w_0$, $v_1 < v_2 < w_0$ or $w_0 < v_2 < v_1$, Corollary~\ref{coro:degree_4_no_nexus} contradicts the fact that $u_i$ is the nexus between $w_0$, $v_1$ and $v_2$.
\end{itemize}
\item If one of the two subpaths has a vertex $v_1$ of degree~4 and the other does not:
\begin{itemize}
\item If the other subpath has no vertices of degree~3 or adjacent vertices to any vertex of degree~3 outside the path, that subpath can be discarded.
\item If the other subpath has a single vertex of degree~3 that is adjacent to $u_i$ and there are no adjacent vertices to a vertex of degree~3 outside the path, that subpath can be discarded.
\item If neither of the above two options occurs, the other subpath has either a vertex $v_2$ of degree~3 that is not adjacent to $u_i$, or a vertex adjacent to some vertex $v_2$ of degree~3 outside the path. In both cases, we have that $u_i$ is the nexus between $w_0$, $v_1$, and $v_2$, and furthermore, $u_i$ is not adjacent to either $w_0$ or $v_2$. But then:
\begin{itemize}
\item If $v_2 < v_1 < w_0$ or $w_0 < v_1 < v_2$, Corollary~\ref{coro:degree_4_no_nexus} contradicts the fact that $u_i$ is the nexus between vertices $w_0$, $v_1$, and $v_2$.
\item If $v_1 < w_0 < v_2$, $v_2 < w_0 < v_1$, $w_0 < v_2 < v_1$, or $v_1 < v_2 < w_0$, Corollary~\ref{coro:degree_3_ady_no_nexus} contradicts the fact that $u_i$ is the nexus between vertices $w_0$, $v_1$, and $v_2$.
\end{itemize}
\end{itemize}
\item If there are no vertices of degree~4 in either subpath:
\begin{itemize}
\item If in either subpath there are no vertices of degree~3, nor vertices adjacent to any vertex of degree~3 outside the path, that subpath can be discarded.
\item If in either of the two subpaths there is a single vertex of degree~3 that is adjacent to $u_i$ and in addition that subpath has no vertices adjacent to any vertex of degree~3 outside the path, that subpath can be discarded.
\item If neither of the two options above happens, in both subpaths there is either a vertex $v_j$ of degree~3 that is not adjacent to $u_i$, or a vertex adjacent to some vertex $v_j$ of degree~3 outside the path. In all cases, we have that $u_i$ is the nexus between $w_0$, $v_1$ and $v_2$ (where $v_1$ and $v_2$ are the $v_j$ that exist for each of the two subpaths) and in addition $u_i$ is not adjacent to any of these 3 vertices. But the Corollary~\ref{coro:degree_3_ady_no_nexus} contradicts the fact that $u_i$ is the nexus between these 3 vertices.
\end{itemize}
\end{itemize}
Thus, we conclude that in all possible cases the path $C_2$ can be constructed. We can continue repeating this process by taking paths $C_3, C_4, \dots$ such that all satisfy the condition in point~1, and for all $i$ it holds that $W_i$ is strictly included in $W_{i-1}$. Eventually, we obtain a path $C_j$ such that $W_j = \emptyset$. Therefore, the path $C_0' = C_j$ satisfies 1. and~2.(a).
\item By the above, if $\pthin(T) = 2$, let $C_0' = v_0, \dots, v_k$ be the path that exists by 1) and 2)a). Let $v_3 \notin C_0'$ with $d(v_3) = 3$, and let $w \in C_0'$ be adjacent to $v_3$. We want to prove that $d(w) = 3$:
\begin{itemize}
\item By a reasoning analogous to the previous point, if any of the subpaths of $w$ in $C_0'$ does not have any vertices of degree $\geq 3$, a new path $C_0''$ could be taken, discarding that subpath.
\item If, on the other hand, both subpaths of $w$ in $C_0'$ have some vertex of degree $\geq 3$, a new path cannot be taken. Therefore, to complete the proof, it is enough to prove that, if in this situation, $d(w) \geq 4$ holds, we arrive at an absurdity: Let $v_1$ and $v_2$ be two vertices of degree greater than or equal to~3 in $C_0'$, one in each subpath. Let $v_i$ (with $i \in \{1, 2, 3\}$) be the intermediate vertex in the order among the~3. We know that $d(v_i) = 3$, since if $d(v_i) \geq 4$, the Corollary~\ref{coro:degree_4_no_nexus} contradicts the fact that $w$ is the nexus between the 3 vertices. Furthermore, by the Corollary~\ref{coro:degree_3_ady_no_nexus}, we know that $v_i$ is adjacent to $w$. \begin{itemize}
\item If $v_i$ and $w$ are in the same class of the partition (for example, $v_i \in V^0$ and $w \in V^0$), and assuming that $v_i < w$ holds, then there is a $v_j < v_i < w$ (if instead, $v_i > w$, we take $v_j > v_i > w$). Since $d(v_i) = 3$, there exists a vertex $x \in V^1$ adjacent to $v_i$. Therefore $v_j$ is smaller than the edge $v_i x$ and $w$ is larger, but then by Proposition~\ref{prop:does_not_cross_edge} there cannot exist a simple path from $v_j$ to $w$ that does not pass through $v_i$ or $x$.
\item If $v_i$ and $w$ are in different classes of the partition (for example, $v_i \in V^1$ and $w \in V^0$), then since $d(w) \geq 4$, there exists another vertex $x$ adjacent to $w$ in $V^1$.
\begin{itemize}
\item If $v_i < x$, then the other two adjacent vertices to $v_i$ ($y$ and $z$) are less than $v_i$. At least one of these two, for example $y$, must be in $V^0$. By Proposition~\ref{prop:no_double_crossing}, it holds that $y < w$. But then $v_j$ that is less than $v_i$, must be less than $y$ (since otherwise, it would be adjacent to $y$ or $v_i$ depending on which class it is in). Then, by Proposition~\ref{prop:does_not_cross_edge}, there cannot exist a simple path from $w$ to $v_j$ that does not pass through $v_i$ or $y$.
\item If $x < v_i$, by a reasoning symmetric to the previous one, by the Proposition~\ref{prop:does_not_cross_edge}, there cannot exist a simple path from $w$ to $v_k$ (where $v_i < v_k$) that does not pass through $v_i$ or a neighbor of $v_i$.
\end{itemize}
\end{itemize}
\end{itemize}
\end{enumerate}
\end{enumerate}
\end{proof}

\subsubsection{Proof of the structural characterization theorem}

Using the previous results, we can prove the main result:

\begin{proof}[Proof of Theorem~\ref{theorem:characterization_pthin2}]
\textbf{$\Rightarrow$)} It follows from Proposition~\ref{prop:pthin2_has_claw}, Proposition~\ref{prop:pthin2_degree_5}, and Proposition~\ref{prop:pthin2_has_path}.

\textbf{$\Leftarrow$)} Let $T = (V, E)$ be a tree satisfying~1., 2., and~3.. We want to prove that $\pthin(T) = 2$. Let $C_0$ be the simple path that exists by~3.. We can take $C_0' = v_1, \dots, v_k$ to be a simple path that arises from $C_0$ by extending it to two leaves $v_1$ and $v_k$, taking it such that there is no other path $C_0''$ that also extends $C_0$ and that satisfies that the number of vertices in $C_0''$ is greater than that of $C_0'$. We want to prove that, if we take $V^0 = \{v_1, \dots, v_k \}$ and $V^1 = V \setminus V^0$, there exists an ordering $\sigma$ strongly consistent with the class partition $S = \{ V^0, V^1 \}$: We start with $\sigma _0 = v_1 < \dots < v_k$ and we want to add the vertices of $V^1$ to the ordering. Since $C_0'$ is a simple path, there are no three vertices in $V^0$ that do not satisfy the proper thinness condition, so the ordering $\sigma _0$ is strongly consistent with the vertices in $V^0$. We traverse in order $v_1, \dots , v_k$ adding to the ordering the vertices of $V^1$ adjacent to each $v_i$. Since $v_1$ and $v_k$ are leaves, there is nothing to add in the first and last steps. Then, for $i$ from $2$ to $k - 1$:
\begin{itemize}
\item If $d(v_i) = 2$, both vertices adjacent to $v_i$ are in $V^0$. Then, there is nothing to add, that is, $\sigma _i = \sigma _{i-1}$.
\item If $d(v_i) = 3$, there is exactly one vertex $w_0 \in V^1$ adjacent to $v_i$. There are two possibilities:
\begin{itemize}
\item If $d(w_0) = 3$, then $w_0$ has two branches $w_1, \dots, w_l$ and $u_1, \dots, u_m$ of vertices in $V^1$. We take $\sigma _i$ by adding just before and just after $v_i$: $u_m < \dots < u_1 < v_i < w_0 < w_1 < \dots < w_l$.
\item Otherwise, starting from $w_0$ there is a simple branch $w_0, \dots, w_l$. We take $\sigma _i$ by adding just after $v_i$: $v_i < w_0 < \dots < w_l$.
\end{itemize}
\item If $d(v_i) = 4$, there exist exactly two vertices $u_0$ and $w_0$ in $V^1$ adjacent to $v_i$, which have simple branches $u_0, \dots, u_m$ and $w_0, \dots, w_l$. We take $\sigma _i$ by adding just before and just after $v_i$: $u_m < \dots < u_0 < v_i < w_0 < \dots < w_l$.
\item If $d(v_i) = s$ with $s \geq 5$, then there are at most two vertices $u_0$ and $w_0$ in $V^1$ adjacent to $v_i$ that have branches to more than one vertex $u_0, \dots, u_m$ and $w_0, \dots, w_l$. Furthermore, there exist $x_0, \dots , x_{s - 4}$ leaves adjacent to $v_i$ that are in $V^1$. We take $\sigma _i$ by adding just before and just after $v_i$: $u_m < \dots < u_0 < v_i < x_0 < \dots < x_{s - 4} < w_0 < \dots < w_l$.
\end{itemize}
\end{proof}

Using the path $C_0$ from the previous proof, we can prove both corollaries of the theorem.

\subsection{Proof of the characterization by minimal forbidden induced subgraphs theorem}

To complete the characterization, we need to prove that the graphs are minimal. To do so, we first prove the following auxiliary proposition:

\begin{proposition}\label{prop:pthin_proper_subgraphs}
Let $T$ be a tree, and let $G$ and $H$ be two trees such that $H$ is a proper subgraph of $T$ and $G$ is a proper subgraph of $H$. Then, $\pthin(G) \leq \pthin(H)$.
\end{proposition}

Now, we prove that they are minimal:

\begin{proof}
To prove that the graphs in the figure are minimal, it suffices to see that for all such graphs, all their connected proper subgraphs (i.e., that they are trees) have proper thinness equal to~2. To show that these subtrees have proper thinness equal to~2, we show a 2-thin proper representation of each one (the vertex order is left to right and the classes are horizontal).

Furthermore, by Proposition~\ref{prop:pthin_proper_subgraphs}, it suffices to prove that, if $T$ is any of the trees in the figure, $\pthin(G) = 2$ for every $G$ connected proper subgraph of $T$ such that $G$ is not itself a proper subgraph of $H$ ($H$ also being a proper subgraph of $T$).
\begin{itemize}
\item $T_0$: The only proper subgraph that is not a proper subgraph of any other graph is:

\input{img/subgraphs/T0.tex}

\item $T_1$: The only proper subgraph that is not a proper subgraph of any other graph is:

\input{img/subgraphs/T1.tex}

\item $T_2^{i,j,k}$: The only proper subgraph that is not a proper subgraph of any other graph is:

\input{img/subgraphs/T2.tex}

\item $T_3^{i,j}$: The only proper subgraphs that are not a proper subgraph of any other graph are:

\input{img/subgraphs/T3.tex}

\item $T_4^i$: The only proper subgraphs that are not a proper subgraph of any other graph are:

\input{img/subgraphs/T4.tex}

\item $T_5^{i,j}$: The only proper subgraphs that are not a proper subgraph of any other graph are:

\input{img/subgraphs/T5.tex}

\end{itemize}
\end{proof}

\subsection{Proof of the proposition about the counterexample $T_A$}

\begin{proof}[Proof of Proposition~\ref{prop:path1}]
Suppose (w.l.o.g.) that $v_0 \in V^0$. Let us see that the subgraphs induced by $V^0$, $V^1$ and $V^2$ are not connected:
\begin{itemize}
\item $V^0$: by Proposition~\ref{prop:same_class_neighbors}, there are at most two vertices adjacent to $v_0$ in $V^0$. Therefore, there exists some of the adjacent vertices to $v_0$ in one of the other classes, let us call that vertex $v_1$, suppose w.l.o.g. that $v_1 \in V^1$. Let $w$ be a vertex in the same subtree of $v_1$ such that $w \in V^0$ (we know it exists by Remark~\ref{obs:one_in_each_class}). But then the only simple path from $w$ to $v_0$ passes through $v_1$, so $T_A[V^0]$ is not connected.
\item $V^1$: By Observation~\ref{obs:one_in_each_class}, we know that there exist two vertices $w_1$ and $w_2$ that belong to two distinct subtrees and are in class $V^1$, but the only simple path between these vertices passes through $v_0$, which belongs to $V^0$, so the subgraph induced by $V^1$ is not connected.
\item $V^2$: equivalent to the previous item, it is proved that the subgraph induced by $V^2$ is not connected.
\end{itemize}
Now let $C_0$ be a simple path. If $C_0$ has at least one vertex in each of the three $T_0$ subtrees that $T_A$ has, then $v_0 \in C_0$ would also hold, and so would its three adjacent vertices, which contradicts that $C_0$ is a simple path.
Therefore, some copies of $T_0$ are contained in $T_A - C_0$. By Theorem~\ref{theorem:forbidden_subtrees}, we know that $\pthin(T_A - C_0) = 3$.
\end{proof}

%% file: img/subgraphs/T0.tex
\begin{figure}[H]
     \centering
     \begin{subfigure}[b]{0.4\textwidth}
         \centering
         \begin{tikzpicture}[scale=0.35]
        	\begin{pgfonlayer}{nodelayer}
        		\node [circle,fill=black,scale=0.55] (0) at (0, 0) {};
        		\node [circle,fill=black,scale=0.55] (1) at (-1.25, -1.75) {};
        		\node [circle,fill=black,scale=0.55] (2) at (1.25, -1.75) {};
        		\node [circle,fill=black,scale=0.55] (3) at (2.25, 0.75) {};
        		\node [circle,fill=black,scale=0.55] (4) at (-2.25, 0.75) {};
        		\node [circle,fill=black,scale=0.55] (5) at (0, 2.25) {};
        		\node [circle,fill=black,scale=0.55] (6) at (-4.5, 1.5) {};
        		\node [circle,fill=black,scale=0.55] (8) at (-2.5, -3.5) {};
        		\node [circle,fill=black,scale=0.55] (9) at (2.5, -3.5) {};
        		\node [circle,fill=black,scale=0.55] (10) at (4.5, 1.5) {};
        	\end{pgfonlayer}
        	\begin{pgfonlayer}{edgelayer}
        		\draw (0.center) to (5.center);
        		\draw (0.center) to (4.center);
        		\draw (4.center) to (6.center);
        		\draw (0.center) to (1.center);
        		\draw (1.center) to (8.center);
        		\draw (0.center) to (2.center);
        		\draw (2.center) to (9.center);
        		\draw (0.center) to (3.center);
        		\draw (3.center) to (10.center);
        	\end{pgfonlayer}
        \end{tikzpicture}
         \caption{Proper subgraph of $T_0$}
     \end{subfigure}
     \hfill
     \begin{subfigure}[b]{0.4\textwidth}
         \centering
         \begin{tikzpicture}[scale=0.35,rotate=270]
        	\begin{pgfonlayer}{nodelayer}
        		\node [circle,fill=black,scale=0.55] (0) at (0, 0) {};
        		\node [circle,fill=black,scale=0.55] (1) at (0, -2) {};
        		\node [circle,fill=black,scale=0.55] (2) at (2, -1) {};
        		\node [circle,fill=black,scale=0.55] (3) at (2, 2) {};
        		\node [circle,fill=black,scale=0.55] (4) at (2, 0.7) {};
        		\node [circle,fill=black,scale=0.55] (5) at (0, 3) {};
        		\node [circle,fill=black,scale=0.55] (7) at (0, 5) {};
        		\node [circle,fill=black,scale=0.55] (8) at (0, -4) {};
        		\node [circle,fill=black,scale=0.55] (9) at (2, -3) {};
        		\node [circle,fill=black,scale=0.55] (10) at (2, 4) {};
        	\end{pgfonlayer}
        	\begin{pgfonlayer}{edgelayer}
        		\draw (0.center) to (5.center);
        		\draw (5.center) to (7.center);
        		\draw (0.center) to (4.center);
        		\draw (0.center) to (1.center);
        		\draw (1.center) to (8.center);
        		\draw (0.center) to (2.center);
        		\draw (2.center) to (9.center);
        		\draw (0.center) to (3.center);
        		\draw (3.center) to (10.center);
        	\end{pgfonlayer}
        \end{tikzpicture}
        
         \caption{Proper 2-thin representation}
     \end{subfigure}

        \caption{Proper subgraph of $T_0$}
        \label{fig:T_0}
\end{figure}
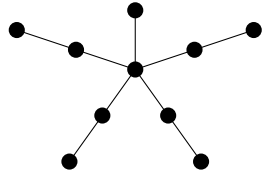
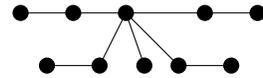

%% file: img/subgraphs/T1.tex
\begin{figure}[H]
     \centering
     \begin{subfigure}[b]{0.4\textwidth}
         \centering
         \begin{tikzpicture}[scale=0.35]
        	\begin{pgfonlayer}{nodelayer}
        		\node [circle,fill=black,scale=0.55] (0) at (0, 0) {};
        		\node [circle,fill=black,scale=0.55] (1) at (2, 0) {};
        		\node [circle,fill=black,scale=0.55] (2) at (0, -2) {};
        		\node [circle,fill=black,scale=0.55] (3) at (-2, 0) {};
        		\node [circle,fill=black,scale=0.55] (4) at (-4, 0) {};
        		\node [circle,fill=black,scale=0.55] (5) at (-3.5, -1.5) {};
        		\node [circle,fill=black,scale=0.55] (6) at (-3.5, 1.5) {};
        		\node [circle,fill=black,scale=0.55] (7) at (4, 0) {};
        		\node [circle,fill=black,scale=0.55] (8) at (3.5, 1.5) {};
        		\node [circle,fill=black,scale=0.55] (9) at (3.5, -1.5) {};
        		\node [circle,fill=black,scale=0.55] (11) at (1.5, -3.5) {};
        		\node [circle,fill=black,scale=0.55] (12) at (-1.5, -3.5) {};
        	\end{pgfonlayer}
        	\begin{pgfonlayer}{edgelayer}
        		\draw (0.center) to (3.center);
        		\draw (3.center) to (4.center);
        		\draw (6.center) to (3.center);
        		\draw (3.center) to (5.center);
        		\draw (0.center) to (2.center);
        		\draw (2.center) to (12.center);
        		\draw (2.center) to (11.center);
        		\draw (0.center) to (1.center);
        		\draw (1.center) to (8.center);
        		\draw (1.center) to (7.center);
        		\draw (1.center) to (9.center);
        	\end{pgfonlayer}
        \end{tikzpicture}
         \caption{Proper subgraph of $T_1$}
     \end{subfigure}
     \hfill
     \begin{subfigure}[b]{0.4\textwidth}
         \centering
         \begin{tikzpicture}[scale=0.35,rotate=270]
        	\begin{pgfonlayer}{nodelayer}
        		\node [circle,fill=black,scale=0.55] (0) at (0, 0) {};
        		\node [circle,fill=black,scale=0.55] (1) at (0, -3) {};
        		\node [circle,fill=black,scale=0.55] (2) at (2, 0.7) {};
        		\node [circle,fill=black,scale=0.55] (3) at (0, 4) {};
        		\node [circle,fill=black,scale=0.55] (4) at (2, 5) {};
        		\node [circle,fill=black,scale=0.55] (5) at (0, 6) {};
        		\node [circle,fill=black,scale=0.55] (6) at (2, 3) {};
        		\node [circle,fill=black,scale=0.55] (7) at (2, -4) {};
        		\node [circle,fill=black,scale=0.55] (8) at (0, -5) {};
        		\node [circle,fill=black,scale=0.55] (9) at (2, -2) {};
        		\node [circle,fill=black,scale=0.55] (11) at (2, -1) {};
        		\node [circle,fill=black,scale=0.55] (12) at (2, 2) {};
        	\end{pgfonlayer}
        	\begin{pgfonlayer}{edgelayer}
        		\draw (0.center) to (3.center);
        		\draw (3.center) to (4.center);
        		\draw (6.center) to (3.center);
        		\draw (3.center) to (5.center);
        		\draw (0.center) to (2.center);
        		\draw (2.center) to (12.center);
        		\draw (2.center) to (11.center);
        		\draw (0.center) to (1.center);
        		\draw (1.center) to (8.center);
        		\draw (1.center) to (7.center);
        		\draw (1.center) to (9.center);
        	\end{pgfonlayer}
        \end{tikzpicture}
        
         \caption{Proper 2-thin representation}
     \end{subfigure}

        \caption{Proper subgraph of $T_1$}
        \label{fig:T_1}
\end{figure}
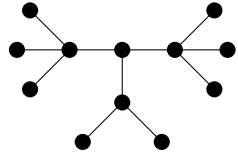
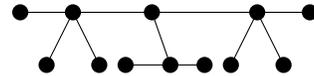

%% file: img/subgraphs/T2.tex
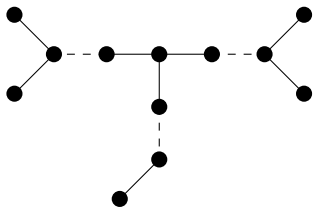
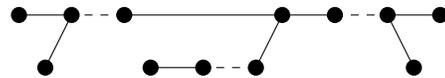
\begin{figure}[H]
     \centering
     \begin{subfigure}[b]{0.4\textwidth}
         \centering
         \begin{tikzpicture}[scale=0.35]
        	\begin{pgfonlayer}{nodelayer}
        		\node [circle,fill=black,scale=0.55] (0) at (0, 0) {};
        		\node [circle,fill=black,scale=0.55] (1) at (-2, 0) {};
        		\node [circle,fill=black,scale=0.55] (2) at (0, -2) {};
        		\node [circle,fill=black,scale=0.55] (3) at (2, 0) {};
        		\node [circle,fill=black,scale=0.55] (4) at (4, 0) {};
        		\node [circle,fill=black,scale=0.55] (6) at (-4, 0) {};
        		\node [circle,fill=black,scale=0.55] (7) at (-5.5, 1.5) {};
        		\node [circle,fill=black,scale=0.55] (8) at (-5.5, -1.5) {};
        		\node [circle,fill=black,scale=0.55] (9) at (5.5, 1.5) {};
        		\node [circle,fill=black,scale=0.55] (10) at (5.5, -1.5) {};
        		\node [circle,fill=black,scale=0.55] (11) at (0, -4) {};
        		\node [circle,fill=black,scale=0.55] (12) at (-1.5, -5.5) {};
        	\end{pgfonlayer}
        	\begin{pgfonlayer}{edgelayer}
        		\draw (0.center) to (1.center);
        		\draw [dashed] (1.center) to (6.center);
        		\draw (6.center) to (7.center);
        		\draw (6.center) to (8.center);
        		\draw (0.center) to (2.center);
        		\draw [dashed] (2.center) to (11.center);
        		\draw (11.center) to (12.center);
        		\draw (0.center) to (3.center);
        		\draw [dashed] (3.center) to (4.center);
        		\draw (4.center) to (9.center);
        		\draw (4.center) to (10.center);
        	\end{pgfonlayer}
        \end{tikzpicture}
         \caption{Proper subgraph of $T_2^{i,j,k}$}
     \end{subfigure}
     \hfill
     \begin{subfigure}[b]{0.4\textwidth}
         \centering
        \begin{tikzpicture}[scale=0.35,rotate=270]
        	\begin{pgfonlayer}{nodelayer}
        		\node [circle,fill=black,scale=0.55] (0) at (0, 0) {};
        		\node [circle,fill=black,scale=0.55] (1) at (0, 2) {};
        		\node [circle,fill=black,scale=0.55] (2) at (0, -6) {};
        		\node [circle,fill=black,scale=0.55] (3) at (2, -1) {};
        		\node [circle,fill=black,scale=0.55] (4) at (2, -3) {};
        		\node [circle,fill=black,scale=0.55] (6) at (0, 4) {};
        		\node [circle,fill=black,scale=0.55] (7) at (0, 6) {};
        		\node [circle,fill=black,scale=0.55] (8) at (2, 5) {};
        		\node [circle,fill=black,scale=0.55] (10) at (2, -5) {};
        		\node [circle,fill=black,scale=0.55] (11) at (0, -8) {};
        		\node [circle,fill=black,scale=0.55] (12) at (0, -10) {};
        		\node [circle,fill=black,scale=0.55] (13) at (2, -9) {};
        	\end{pgfonlayer}
        	\begin{pgfonlayer}{edgelayer}
        		\draw (0.center) to (1.center);
        		\draw [dashed] (1.center) to (6.center);
        		\draw (6.center) to (7.center);
        		\draw (6.center) to (8.center);
        		\draw (0.center) to (2.center);
        		\draw [dashed] (2.center) to (11.center);
        		\draw (11.center) to (12.center);
        		\draw (11.center) to (13.center);
        		\draw (0.center) to (3.center);
        		\draw [dashed] (3.center) to (4.center);
        		\draw (4.center) to (10.center);
        	\end{pgfonlayer}
        \end{tikzpicture}

         \caption{Proper 2-thin representation}
     \end{subfigure}

        \caption{Proper subgraph of $T_2^{i,j,k}$}
        \label{fig:T_2}
\end{figure}

%% file: img/subgraphs/T3.tex
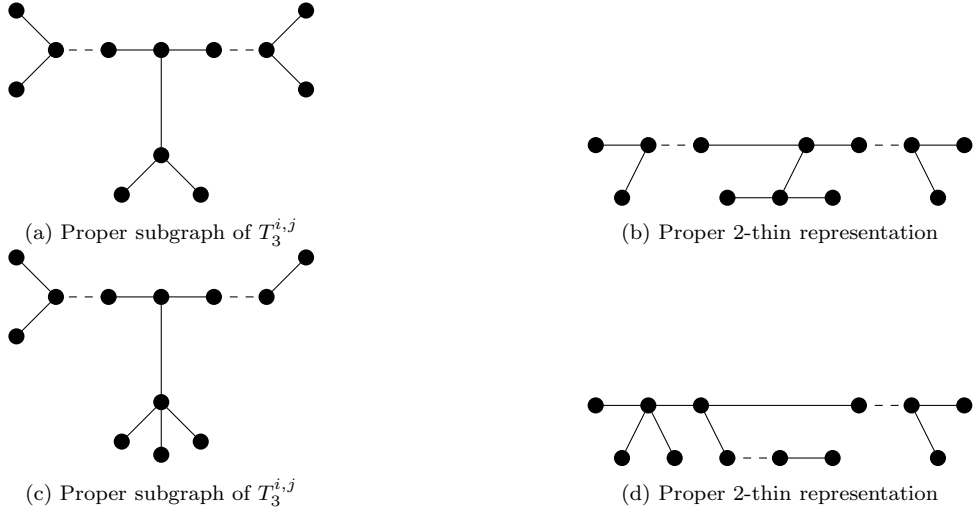
\begin{figure}[H]
     \centering
     \begin{subfigure}[b]{0.4\textwidth}
         \centering
         \begin{tikzpicture}[scale=0.35]
        	\begin{pgfonlayer}{nodelayer}
        		\node [circle,fill=black,scale=0.55] (0) at (0, 0) {};
        		\node [circle,fill=black,scale=0.55] (1) at (0, -4) {};
        		\node [circle,fill=black,scale=0.55] (3) at (-1.5, -5.5) {};
        		\node [circle,fill=black,scale=0.55] (4) at (1.5, -5.5) {};
        		\node [circle,fill=black,scale=0.55] (5) at (2, 0) {};
        		\node [circle,fill=black,scale=0.55] (6) at (-2, 0) {};
        		\node [circle,fill=black,scale=0.55] (7) at (4, 0) {};
        		\node [circle,fill=black,scale=0.55] (8) at (5.5, 1.5) {};
        		\node [circle,fill=black,scale=0.55] (9) at (5.5, -1.5) {};
        		\node [circle,fill=black,scale=0.55] (10) at (-4, 0) {};
        		\node [circle,fill=black,scale=0.55] (11) at (-5.5, 1.5) {};
        		\node [circle,fill=black,scale=0.55] (12) at (-5.5, -1.5) {};
        	\end{pgfonlayer}
        	\begin{pgfonlayer}{edgelayer}
        		\draw (0.center) to (6.center);
        		\draw [dashed] (6.center) to (10.center);
        		\draw (10.center) to (11.center);
        		\draw (10.center) to (12.center);
        		\draw (0.center) to (1.center);
        		\draw (1.center) to (3.center);
        		\draw (1.center) to (4.center);
        		\draw (0.center) to (5.center);
        		\draw [dashed] (5.center) to (7.center);
        		\draw (7.center) to (8.center);
        		\draw (7.center) to (9.center);
        	\end{pgfonlayer}
        \end{tikzpicture}
         \caption{Proper subgraph of $T_3^{i,j}$}
     \end{subfigure}
     \hfill
     \begin{subfigure}[b]{0.4\textwidth}
         \centering
        \begin{tikzpicture}[scale=0.35,rotate=270]
        	\begin{pgfonlayer}{nodelayer}
        		\node [circle,fill=black,scale=0.55] (0) at (0, 0) {};
        		\node [circle,fill=black,scale=0.55] (1) at (2, -1) {};
        		\node [circle,fill=black,scale=0.55] (3) at (2, -3) {};
        		\node [circle,fill=black,scale=0.55] (4) at (2, 1) {};
        		\node [circle,fill=black,scale=0.55] (5) at (0, 2) {};
        		\node [circle,fill=black,scale=0.55] (6) at (0, -4) {};
        		\node [circle,fill=black,scale=0.55] (7) at (0, 4) {};
        		\node [circle,fill=black,scale=0.55] (8) at (0, 6) {};
        		\node [circle,fill=black,scale=0.55] (9) at (2, 5) {};
        		\node [circle,fill=black,scale=0.55] (10) at (0, -6) {};
        		\node [circle,fill=black,scale=0.55] (11) at (0, -8) {};
        		\node [circle,fill=black,scale=0.55] (12) at (2, -7) {};
        	\end{pgfonlayer}
        	\begin{pgfonlayer}{edgelayer}
        		\draw (0.center) to (6.center);
        		\draw [dashed] (6.center) to (10.center);
        		\draw (10.center) to (11.center);
        		\draw (10.center) to (12.center);
        		\draw (0.center) to (1.center);
        		\draw (1.center) to (3.center);
        		\draw (1.center) to (4.center);
        		\draw (0.center) to (5.center);
        		\draw [dashed] (5.center) to (7.center);
        		\draw (7.center) to (8.center);
        		\draw (7.center) to (9.center);
        	\end{pgfonlayer}
        \end{tikzpicture}

         \caption{Proper 2-thin representation}
     \end{subfigure}
     
     \centering
     \begin{subfigure}[b]{0.4\textwidth}
         \centering
         \begin{tikzpicture}[scale=0.35]
        	\begin{pgfonlayer}{nodelayer}
        		\node [circle,fill=black,scale=0.55] (0) at (0, 0) {};
        		\node [circle,fill=black,scale=0.55] (1) at (0, -4) {};
        		\node [circle,fill=black,scale=0.55] (2) at (0, -6) {};
        		\node [circle,fill=black,scale=0.55] (3) at (-1.5, -5.5) {};
        		\node [circle,fill=black,scale=0.55] (4) at (1.5, -5.5) {};
        		\node [circle,fill=black,scale=0.55] (5) at (2, 0) {};
        		\node [circle,fill=black,scale=0.55] (6) at (-2, 0) {};
        		\node [circle,fill=black,scale=0.55] (7) at (4, 0) {};
        		\node [circle,fill=black,scale=0.55] (8) at (5.5, 1.5) {};
        		\node [circle,fill=black,scale=0.55] (10) at (-4, 0) {};
        		\node [circle,fill=black,scale=0.55] (11) at (-5.5, 1.5) {};
        		\node [circle,fill=black,scale=0.55] (12) at (-5.5, -1.5) {};
        	\end{pgfonlayer}
        	\begin{pgfonlayer}{edgelayer}
        		\draw (0.center) to (6.center);
        		\draw [dashed] (6.center) to (10.center);
        		\draw (10.center) to (11.center);
        		\draw (10.center) to (12.center);
        		\draw (0.center) to (1.center);
        		\draw (1.center) to (3.center);
        		\draw (1.center) to (2.center);
        		\draw (1.center) to (4.center);
        		\draw (0.center) to (5.center);
        		\draw [dashed] (5.center) to (7.center);
        		\draw (7.center) to (8.center);
        	\end{pgfonlayer}
        \end{tikzpicture}
         \caption{Proper subgraph of $T_3^{i,j}$}
     \end{subfigure}
     \hfill
     \begin{subfigure}[b]{0.4\textwidth}
         \centering
        \begin{tikzpicture}[scale=0.35,rotate=270]
        	\begin{pgfonlayer}{nodelayer}
        		\node [circle,fill=black,scale=0.55] (0) at (0, 0) {};
        		\node [circle,fill=black,scale=0.55] (1) at (0, -2) {};
        		\node [circle,fill=black,scale=0.55] (2) at (0, -4) {};
        		\node [circle,fill=black,scale=0.55] (3) at (2, -1) {};
        		\node [circle,fill=black,scale=0.55] (4) at (2, -3) {};
        		\node [circle,fill=black,scale=0.55] (5) at (2, 1) {};
        		\node [circle,fill=black,scale=0.55] (6) at (0, 6) {};
        		\node [circle,fill=black,scale=0.55] (7) at (2, 3) {};
        		\node [circle,fill=black,scale=0.55] (8) at (2, 5) {};
        		\node [circle,fill=black,scale=0.55] (10) at (0, 8) {};
        		\node [circle,fill=black,scale=0.55] (11) at (0, 10) {};
        		\node [circle,fill=black,scale=0.55] (12) at (2, 9) {};
        	\end{pgfonlayer}
        	\begin{pgfonlayer}{edgelayer}
        		\draw (0.center) to (6.center);
        		\draw [dashed] (6.center) to (10.center);
        		\draw (10.center) to (11.center);
        		\draw (10.center) to (12.center);
        		\draw (0.center) to (1.center);
        		\draw (1.center) to (3.center);
        		\draw (1.center) to (2.center);
        		\draw (1.center) to (4.center);
        		\draw (0.center) to (5.center);
        		\draw [dashed] (5.center) to (7.center);
        		\draw (7.center) to (8.center);
        	\end{pgfonlayer}
        \end{tikzpicture}

         \caption{Proper 2-thin representation}
     \end{subfigure}

        \caption{Proper subgraphs of $T_3^{i,j}$}
        \label{fig:T_3}
\end{figure}

%% file: img/subgraphs/T4.tex
\begin{figure}[H]
     \centering
     \begin{subfigure}[b]{0.4\textwidth}
         \centering
         \begin{tikzpicture}[scale=0.35]
        	\begin{pgfonlayer}{nodelayer}
        		\node [circle,fill=black,scale=0.55] (0) at (0, 0) {};
        		\node [circle,fill=black,scale=0.55] (1) at (4, 0) {};
        		\node [circle,fill=black,scale=0.55] (3) at (5.5, 1.5) {};
        		\node [circle,fill=black,scale=0.55] (4) at (5.5, -1.5) {};
        		\node [circle,fill=black,scale=0.55] (5) at (-4, 0) {};
        		\node [circle,fill=black,scale=0.55] (6) at (-6, 0) {};
        		\node [circle,fill=black,scale=0.55] (7) at (-5.5, 1.5) {};
        		\node [circle,fill=black,scale=0.55] (8) at (-5.5, -1.5) {};
        		\node [circle,fill=black,scale=0.55] (9) at (0, -2) {};
        		\node [circle,fill=black,scale=0.55] (10) at (0, -4) {};
        		\node [circle,fill=black,scale=0.55] (11) at (-1.5, -5.5) {};
        		\node [circle,fill=black,scale=0.55] (12) at (1.5, -5.5) {};
        	\end{pgfonlayer}
        	\begin{pgfonlayer}{edgelayer}
        		\draw (0.center) to (5.center);
        		\draw (5.center) to (7.center);
        		\draw (6.center) to (5.center);
        		\draw (5.center) to (8.center);
        		\draw (9.center) to (0.center);
        		\draw [dashed] (9.center) to (10.center);
        		\draw (10.center) to (11.center);
        		\draw (10.center) to (12.center);
        		\draw (0.center) to (1.center);
        		\draw (1.center) to (3.center);
        		\draw (1.center) to (4.center);
        	\end{pgfonlayer}
        \end{tikzpicture}
         \caption{Proper subgraph of $T_4^{i}$}
     \end{subfigure}
     \hfill
     \begin{subfigure}[b]{0.4\textwidth}
         \centering
        \begin{tikzpicture}[scale=0.35,rotate=270]
        	\begin{pgfonlayer}{nodelayer}
        		\node [circle,fill=black,scale=0.55] (0) at (0, 0) {};
        		\node [circle,fill=black,scale=0.55] (1) at (2, 1) {};
        		\node [circle,fill=black,scale=0.55] (3) at (2, 2.5) {};
        		\node [circle,fill=black,scale=0.55] (4) at (2, -1) {};
        		\node [circle,fill=black,scale=0.55] (5) at (0, 5) {};
        		\node [circle,fill=black,scale=0.55] (6) at (0, 7) {};
        		\node [circle,fill=black,scale=0.55] (7) at (2, 6) {};
        		\node [circle,fill=black,scale=0.55] (8) at (2, 4) {};
        		\node [circle,fill=black,scale=0.55] (9) at (0, -2) {};
        		\node [circle,fill=black,scale=0.55] (10) at (0, -4) {};
        		\node [circle,fill=black,scale=0.55] (11) at (0, -6) {};
        		\node [circle,fill=black,scale=0.55] (12) at (2, -5) {};
        	\end{pgfonlayer}
        	\begin{pgfonlayer}{edgelayer}
        		\draw (0.center) to (5.center);
        		\draw (5.center) to (7.center);
        		\draw (6.center) to (5.center);
        		\draw (5.center) to (8.center);
        		\draw (9.center) to (0.center);
        		\draw [dashed] (9.center) to (10.center);
        		\draw (10.center) to (11.center);
        		\draw (10.center) to (12.center);
        		\draw (0.center) to (1.center);
        		\draw (1.center) to (3.center);
        		\draw (1.center) to (4.center);
        	\end{pgfonlayer}
        \end{tikzpicture}

         \caption{Proper 2-thin representation}
     \end{subfigure}
     
     \centering
     \begin{subfigure}[b]{0.4\textwidth}
         \centering
         \begin{tikzpicture}[scale=0.35]
        	\begin{pgfonlayer}{nodelayer}
        		\node [circle,fill=black,scale=0.55] (0) at (0, 0) {};
        		\node [circle,fill=black,scale=0.55] (1) at (4, 0) {};
        		\node [circle,fill=black,scale=0.55] (2) at (6, 0) {};
        		\node [circle,fill=black,scale=0.55] (3) at (5.5, 1.5) {};
        		\node [circle,fill=black,scale=0.55] (4) at (5.5, -1.5) {};
        		\node [circle,fill=black,scale=0.55] (5) at (-4, 0) {};
        		\node [circle,fill=black,scale=0.55] (6) at (-6, 0) {};
        		\node [circle,fill=black,scale=0.55] (7) at (-5.5, 1.5) {};
        		\node [circle,fill=black,scale=0.55] (8) at (-5.5, -1.5) {};
        		\node [circle,fill=black,scale=0.55] (9) at (0, -2) {};
        		\node [circle,fill=black,scale=0.55] (10) at (0, -4) {};
        		\node [circle,fill=black,scale=0.55] (11) at (-1.5, -5.5) {};
        	\end{pgfonlayer}
        	\begin{pgfonlayer}{edgelayer}
        		\draw (0.center) to (5.center);
        		\draw (5.center) to (7.center);
        		\draw (6.center) to (5.center);
        		\draw (5.center) to (8.center);
        		\draw (9.center) to (0.center);
        		\draw [dashed] (9.center) to (10.center);
        		\draw (10.center) to (11.center);
        		\draw (0.center) to (1.center);
        		\draw (1.center) to (3.center);
        		\draw (1.center) to (2.center);
        		\draw (1.center) to (4.center);
        	\end{pgfonlayer}
        \end{tikzpicture}
         \caption{Proper subgraph of $T_4^{i}$}
     \end{subfigure}
     \hfill
     \begin{subfigure}[b]{0.4\textwidth}
         \centering
        \begin{tikzpicture}[scale=0.35,rotate=270]
        	\begin{pgfonlayer}{nodelayer}
        		\node [circle,fill=black,scale=0.55] (0) at (0, 0) {};
        		\node [circle,fill=black,scale=0.55] (1) at (0, 2) {};
        		\node [circle,fill=black,scale=0.55] (2) at (2, 3) {};
        		\node [circle,fill=black,scale=0.55] (3) at (0, 4) {};
        		\node [circle,fill=black,scale=0.55] (4) at (2, 1) {};
        		\node [circle,fill=black,scale=0.55] (5) at (0, -6) {};
        		\node [circle,fill=black,scale=0.55] (6) at (2, -7) {};
        		\node [circle,fill=black,scale=0.55] (7) at (2, -8.5) {};
        		\node [circle,fill=black,scale=0.55] (8) at (0, -9) {};
        		\node [circle,fill=black,scale=0.55] (9) at (2, -1) {};
        		\node [circle,fill=black,scale=0.55] (10) at (2, -3) {};
        		\node [circle,fill=black,scale=0.55] (11) at (2, -5) {};
        	\end{pgfonlayer}
        	\begin{pgfonlayer}{edgelayer}
        		\draw (0.center) to (5.center);
        		\draw (5.center) to (7.center);
        		\draw (6.center) to (5.center);
        		\draw (5.center) to (8.center);
        		\draw (9.center) to (0.center);
        		\draw [dashed] (9.center) to (10.center);
        		\draw (10.center) to (11.center);
        		\draw (0.center) to (1.center);
        		\draw (1.center) to (3.center);
        		\draw (1.center) to (2.center);
        		\draw (1.center) to (4.center);
        	\end{pgfonlayer}
        \end{tikzpicture}

         \caption{Proper 2-thin representation}
     \end{subfigure}

        \caption{Proper subgraphs of $T_4^{i}$}
        \label{fig:T_4}
\end{figure}
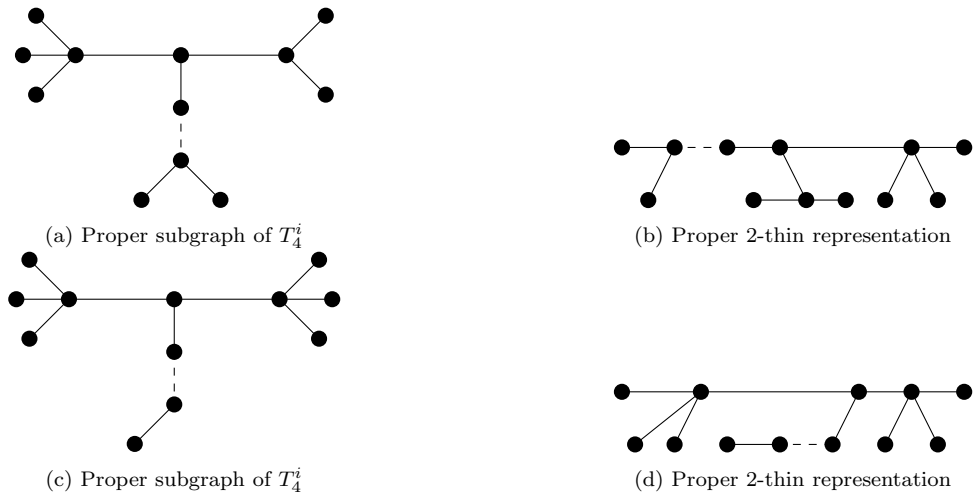

%% file: img/subgraphs/T5.tex
\begin{figure}[H]
     \centering
     \begin{subfigure}[b]{0.4\textwidth}
         \centering
         \begin{tikzpicture}[scale=0.35]
        	\begin{pgfonlayer}{nodelayer}
        		\node [circle,fill=black,scale=0.55] (0) at (0, 0) {};
        		\node [circle,fill=black,scale=0.55] (1) at (4, 0) {};
        		\node [circle,fill=black,scale=0.55] (2) at (5.5, 1.5) {};
        		\node [circle,fill=black,scale=0.55] (3) at (5.5, -1.5) {};
        		\node [circle,fill=black,scale=0.55] (4) at (-4, 0) {};
        		\node [circle,fill=black,scale=0.55] (5) at (-5.5, 1.5) {};
        		\node [circle,fill=black,scale=0.55] (6) at (-5.5, -1.5) {};
        		\node [circle,fill=black,scale=0.55] (7) at (0, -2) {};
        		\node [circle,fill=black,scale=0.55] (8) at (-1.5, -3.5) {};
        		\node [circle,fill=black,scale=0.55] (10) at (0, 2) {};
        	\end{pgfonlayer}
        	\begin{pgfonlayer}{edgelayer}
        		\draw [dashed] (0.center) to (4.center);
        		\draw (5.center) to (4.center);
        		\draw (4.center) to (6.center);
        		\draw (0.center) to (7.center);
        		\draw (7.center) to (8.center);
        		\draw [dashed] (0.center) to (1.center);
        		\draw (1.center) to (2.center);
        		\draw (1.center) to (3.center);
        		\draw (0.center) to (10.center);
        	\end{pgfonlayer}
         \end{tikzpicture}
         \caption{Proper subgraph of $T_5^{i,j}$}
     \end{subfigure}
     \hfill
     \begin{subfigure}[b]{0.4\textwidth}
         \centering
         \begin{tikzpicture}[scale=0.35,rotate=270]
        	\begin{pgfonlayer}{nodelayer}
        		\node [circle,fill=black,scale=0.55] (0) at (0, 0) {};
        		\node [circle,fill=black,scale=0.55] (1) at (0, 2) {};
        		\node [circle,fill=black,scale=0.55] (2) at (0, 4) {};
        		\node [circle,fill=black,scale=0.55] (3) at (2, 3) {};
        		\node [circle,fill=black,scale=0.55] (4) at (0, -4) {};
        		\node [circle,fill=black,scale=0.55] (5) at (2, -5) {};
        		\node [circle,fill=black,scale=0.55] (6) at (0, -6) {};
        		\node [circle,fill=black,scale=0.55] (7) at (2, -1) {};
        		\node [circle,fill=black,scale=0.55] (9) at (2, -3) {};
        		\node [circle,fill=black,scale=0.55] (10) at (2, 1) {};
        	\end{pgfonlayer}
        	\begin{pgfonlayer}{edgelayer}
        		\draw [dashed] (0.center) to (4.center);
        		\draw (5.center) to (4.center);
        		\draw (4.center) to (6.center);
        		\draw (0.center) to (7.center);
        		\draw (7.center) to (9.center);
        		\draw [dashed] (0.center) to (1.center);
        		\draw (1.center) to (2.center);
        		\draw (1.center) to (3.center);
        		\draw (0.center) to (10.center);
        	\end{pgfonlayer}
        \end{tikzpicture}

         \caption{Proper 2-thin representation}
     \end{subfigure}
     \hfill
     \begin{subfigure}[b]{0.4\textwidth}
         \centering
         \begin{tikzpicture}[scale=0.35]
        	\begin{pgfonlayer}{nodelayer}
        		\node [circle,fill=black,scale=0.55] (0) at (0, 0) {};
        		\node [circle,fill=black,scale=0.55] (1) at (4, 0) {};
        		\node [circle,fill=black,scale=0.55] (2) at (5.5, 1.5) {};
        		\node [circle,fill=black,scale=0.55] (4) at (-4, 0) {};
        		\node [circle,fill=black,scale=0.55] (5) at (-5.5, 1.5) {};
        		\node [circle,fill=black,scale=0.55] (6) at (-5.5, -1.5) {};
        		\node [circle,fill=black,scale=0.55] (7) at (0, -2) {};
        		\node [circle,fill=black,scale=0.55] (8) at (-1.5, -3.5) {};
        		\node [circle,fill=black,scale=0.55] (9) at (1.5, -3.5) {};
        		\node [circle,fill=black,scale=0.55] (10) at (0, 2) {};
        	\end{pgfonlayer}
        	\begin{pgfonlayer}{edgelayer}
        		\draw [dashed] (0.center) to (4.center);
        		\draw (5.center) to (4.center);
        		\draw (4.center) to (6.center);
        		\draw (0.center) to (7.center);
        		\draw (7.center) to (8.center);
        		\draw (7.center) to (9.center);
        		\draw [dashed] (0.center) to (1.center);
        		\draw (1.center) to (2.center);
        		\draw (0.center) to (10.center);
        	\end{pgfonlayer}
        \end{tikzpicture}
         \caption{Proper subgraph of $T_5^{i,j}$}
     \end{subfigure}
     \hfill
     \begin{subfigure}[b]{0.4\textwidth}
         \centering
         \begin{tikzpicture}[scale=0.35,rotate=270]
        	\begin{pgfonlayer}{nodelayer}
        		\node [circle,fill=black,scale=0.55] (0) at (0, 0) {};
        		\node [circle,fill=black,scale=0.55] (1) at (0, 2) {};
        		\node [circle,fill=black,scale=0.55] (2) at (0, 4) {};
        		\node [circle,fill=black,scale=0.55] (3) at (2, 3) {};
        		\node [circle,fill=black,scale=0.55] (4) at (2, -1) {};
        		\node [circle,fill=black,scale=0.55] (5) at (2, -3) {};
        		\node [circle,fill=black,scale=0.55] (7) at (0, -4) {};
        		\node [circle,fill=black,scale=0.55] (8) at (0, -6) {};
        		\node [circle,fill=black,scale=0.55] (9) at (2, -5) {};
        		\node [circle,fill=black,scale=0.55] (10) at (2, 1) {};
        	\end{pgfonlayer}
        	\begin{pgfonlayer}{edgelayer}
        		\draw [dashed] (0.center) to (4.center);
        		\draw (5.center) to (4.center);
        		\draw (0.center) to (7.center);
        		\draw (7.center) to (8.center);
        		\draw (7.center) to (9.center);
        		\draw [dashed] (0.center) to (1.center);
        		\draw (1.center) to (2.center);
        		\draw (1.center) to (3.center);
        		\draw (0.center) to (10.center);
        	\end{pgfonlayer}
        \end{tikzpicture}

         \caption{Proper 2-thin representation}
     \end{subfigure}
     \hfill
     \begin{subfigure}[b]{0.4\textwidth}
         \centering
         \begin{tikzpicture}[scale=0.35]
        	\begin{pgfonlayer}{nodelayer}
        		\node [circle,fill=black,scale=0.55] (0) at (0, 0) {};
        		\node [circle,fill=black,scale=0.55] (1) at (4, 0) {};
        		\node [circle,fill=black,scale=0.55] (2) at (5.5, 1.5) {};
        		\node [circle,fill=black,scale=0.55] (3) at (5.5, -1.5) {};
        		\node [circle,fill=black,scale=0.55] (4) at (-4, 0) {};
        		\node [circle,fill=black,scale=0.55] (5) at (-5.5, 1.5) {};
        		\node [circle,fill=black,scale=0.55] (6) at (-5.5, -1.5) {};
        		\node [circle,fill=black,scale=0.55] (7) at (0, -2) {};
        		\node [circle,fill=black,scale=0.55] (8) at (-1.5, -3.5) {};
        		\node [circle,fill=black,scale=0.55] (9) at (1.5, -3.5) {};
        	\end{pgfonlayer}
        	\begin{pgfonlayer}{edgelayer}
        		\draw [dashed] (0.center) to (4.center);
        		\draw (5.center) to (4.center);
        		\draw (4.center) to (6.center);
        		\draw (0.center) to (7.center);
        		\draw (7.center) to (8.center);
        		\draw (7.center) to (9.center);
        		\draw [dashed] (0.center) to (1.center);
        		\draw (1.center) to (2.center);
        		\draw (1.center) to (3.center);
        	\end{pgfonlayer}
        \end{tikzpicture}

         \caption{Proper subgraph of $T_5^{i,j}$}
     \end{subfigure}
     \hfill
     \begin{subfigure}[b]{0.4\textwidth}
         \centering
         \begin{tikzpicture}[scale=0.35,rotate=270]
        	\begin{pgfonlayer}{nodelayer}
        		\node [circle,fill=black,scale=0.55] (0) at (0, 0) {};
        		\node [circle,fill=black,scale=0.55] (1) at (0, 4) {};
        		\node [circle,fill=black,scale=0.55] (2) at (0, 6) {};
        		\node [circle,fill=black,scale=0.55] (3) at (2, 5) {};
        		\node [circle,fill=black,scale=0.55] (4) at (2, 1) {};
        		\node [circle,fill=black,scale=0.55] (5) at (2, 3) {};
        		\node [circle,fill=black,scale=0.55] (6) at (2, -1) {};
        		\node [circle,fill=black,scale=0.55] (7) at (0, -2) {};
        		\node [circle,fill=black,scale=0.55] (8) at (0, -4) {};
        		\node [circle,fill=black,scale=0.55] (9) at (2, -3) {};
        	\end{pgfonlayer}
        	\begin{pgfonlayer}{edgelayer}
        		\draw (0.center) to (4.center);
        		\draw (5.center) to (4.center);
        		\draw (4.center) to (6.center);
        		\draw [dashed] (0.center) to (7.center);
        		\draw (7.center) to (8.center);
        		\draw (7.center) to (9.center);
        		\draw [dashed] (0.center) to (1.center);
        		\draw (1.center) to (2.center);
        		\draw (1.center) to (3.center);
        	\end{pgfonlayer}
        \end{tikzpicture}

         \caption{Proper 2-thin representation}
     \end{subfigure}

        \caption{Proper subgraphs of $T_5^{i,j}$}
        \label{fig:T_5}
\end{figure}
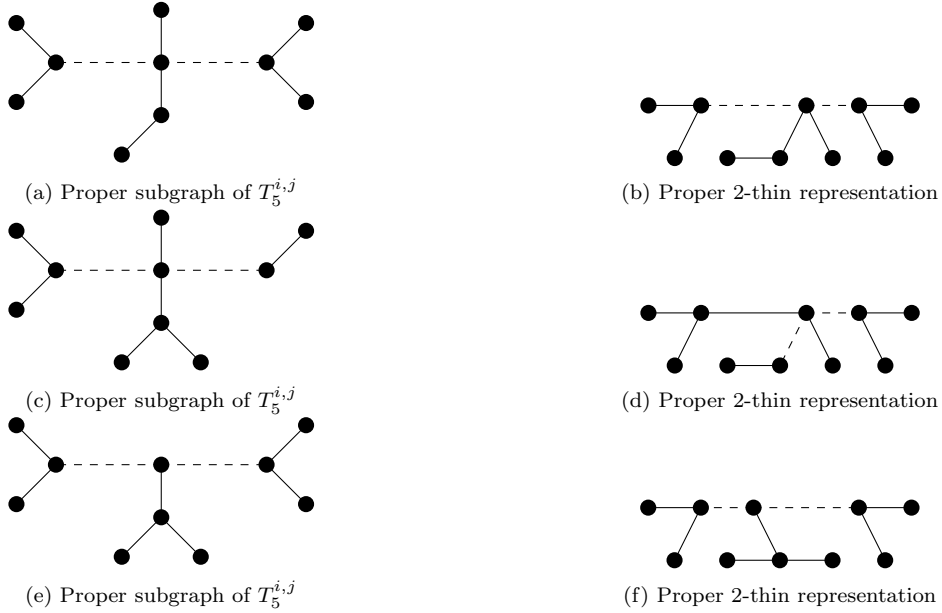